\documentclass[12pt]{article}
\usepackage{amsfonts}
\usepackage{amssymb}
\usepackage{mathrsfs}
\usepackage{srcltx}
\usepackage{amsmath}
\usepackage{amsthm}
\usepackage{amstext}
\usepackage{amsopn}
\usepackage{graphicx}
\usepackage{bm}
\newtheorem{theorem}{Theorem}[section]
\newtheorem{lemma}[theorem]{Lemma}

\theoremstyle{definition}

\theoremstyle{remark}
\newtheorem{remark}[theorem]{Remark}

\numberwithin{equation}{section}

\newcommand{\ba}{\begin{array}}
\newcommand{\ea}{\end{array}}
\newcommand{\f}{\frac}

\newcommand{\la}{\lambda}

\newcommand{\ds}{\displaystyle}

\begin{document}
\date{}
\title{ \bf\large{Hopf bifurcation of a delayed single population model with patch structure}\thanks{This research is supported by the National Natural Science Foundation of China (No 11771109).}}
\author{Shanshan Chen\textsuperscript{1}\footnote{Corresponding Author, Email: chenss@hit.edu.cn},\ \ Zuolin Shen\textsuperscript{1}\footnote{Email: mathust\_lin@foxmail.com},\ \ Junjie Wei\textsuperscript{1,2}\footnote{Email: weijj@hit.edu.cn}\ \
 \\
{\small \textsuperscript{1} Department of Mathematics, Harbin Institute of Technology,\hfill{\ }}\\
\ \ {\small Weihai, Shandong, 264209, P.R.China.\hfill{\ }}\\
{\small \textsuperscript{2} School of Mathematics and Data, Foshan University,\hfill{\ }}\\
\ \ {\small Foshan, Guangdong, 528000, P.R.China.\hfill{\ }}\\
}
\maketitle

\begin{abstract}
In this paper, we show the existence of Hopf bifurcation of a delayed single population model with patch structure.
The effect of the dispersal rate on the Hopf bifurcation is considered. 
Especially, if each patch is favorable for the species, we show that when the dispersal rate tends to zero, the limit of the Hopf bifurcation value is the minimum of the \lq\lq local\rq\rq~Hopf bifurcation values over all patches. On the other hand, when the dispersal rate tends to infinity, the Hopf bifurcation value tends to that of the \lq\lq average\rq\rq~model.

\end{abstract}

\section{Introduction}
To understand the variation of the population densities in ecology,
various models described by delayed differential equations or delayed partial differential equations have been built and analyzed.
For example, the following classical Hutchinson's equation was used to model the growth of a single species:
\begin{equation}\label{class}
\dot u=u(m-u(t-r)),\;\;t>0,
\end{equation}
where $u$ represents the population density in time $t$, $m$ is the intrinsic growth rate, and $r$ represents the maturation time.
It is well known that large time delay $r$ could induce oscillation through Hopf bifurcation, see \cite{KuangY}.

The effect of spatial
environment was not considered in model \eqref{class}. If space is regarded as a continuous variable, one can obtain the following
diffusive Hutchinson's equation:
\begin{equation}\label{dclass}
\displaystyle\frac{\partial u}{\partial t} =d\Delta u+u(m(x)-u(t-r)),\;\;t>0.
\end{equation}
Here the dispersal is completely random, and there exist extensive results on the stability and Hopf bifurcation of model \eqref{dclass} if $m(x)\equiv m$ for $m>0$.
For the homogeneous Neumann boundary condition, Memory \cite{Memory} and Yoshida \cite{Yoshida} obtained the stability of the spatially homogeneous steady state and the existence of Hopf bifurcation. For the homogeneous Dirichlet boundary condition, the positive steady state of model \eqref{dclass} is spatially heterogeneous. Busenberg and Huang \cite{Busenberg} firstly studied the Hopf bifurcation near such spatially heterogeneous positive steady state, and the implicit function theorem and some perturbation arguments were used to show the existence of Hopf bifurcation. This method could also be used to study the Hopf bifurcation for some other population models under the homogeneous Dirichlet boundary conditions, see \cite{ChenShi2012,ChenYu2016,Guo2015,Guo2017,GuoYan2016,HuYuan2011,SuWeiShi2009,SuWeiShi2012,YanLi2010,YanLi2012} and the references therein. If $m(x)$ is spatially heterogenous, the effect of spatial
heterogeneity on model \eqref{dclass} was investigated in \cite{LouY2006} for $\tau=0$, and delay induced Hopf bifurcation was considered
in \cite{ShiShi2019}. Moreover,
the advection term, which indicates movements towards better quality habitat or refers to unidirectional bias of movements, was also taken into consideration for model \eqref{dclass} by many researchers. The dynamics of model \eqref{dclass} with the advection term could be found in \cite{Belgacem1995,CantrellCosner2003,CosnerLou2003,Lam2011,LouLutscher2014,LouZhou2015,Speirs2001,Vasilyeva2019,Vasilyeva2010,ZhouZhao2018} and the references therein for $\tau=0$, and delay induced Hopf bifurcation could be found in \cite{ChenLouWei,ChenWeiZhang}.

In a discrete spatial setting, one could obtain the following patch model:
\begin{equation}\label{patc}
\begin{cases}
\displaystyle\frac{d u_j}{dt}=d\sum_{k=1}^nd_{jk} u_k+u_j(m_j-u_j(t-r)),&t>0,\;\;j=1,\dots,n,\\
\bm u(t)=\bm \psi(t)\ge\bm 0,&t\in [-r,0],
\end{cases}
\end{equation}
where $n\ge 2$, $\bm u=(u_1,\dots, u_n)^T$ and $d_{jj}= -\ds\sum_{k\ne j}d_{kj}$ for any $j=1,\dots,n$.
Here $u_j$ denotes the population density in patch $j$ and time $t$;
$d>0$ represents the dispersal rate of the population; $m_j$ is the intrinsic growth rate in patch $j$, $m_j>0$ if the patch $j$ is favorable for the species, and $m_j\le0$ if the patch $j$ is unfavorable for the species;
$r\ge0$ represents the maturation time of the population; $d_{jk}(j\ne k)\ge0$ denotes the degree of the movements from patch
$k$ to patch $j$, and $d_{jj}$ denotes the outgoing degree of the movements of patch $j$.
In this paper, we will consider the existence of Hopf bifurcation of model \eqref{patc}. The following two assumptions are imposed throughout the paper:
\begin{enumerate}
\item [($A_1$)] The connectivity matrix $D:=(d_{jk})_{n\times n}$ is symmetric, irreducible and quasi-positive;
\item [($A_2$)] $H^+=\{j\in 1,\dots,n: m_j>0\}\ne \emptyset$.
\end{enumerate}
Here the symmetry assumption of $D$ could mimic the random
diffusion of species, which means that the
per capita rate of individuals entering patch $j$ from patch
$k$ is equal to that entering patch $k$ from patch $j$. Then
$d_{jj}= -\ds\sum_{k\ne j}d_{kj}=-\ds\sum_{k\ne j}d_{jk}$ for any $j=1,\dots,n$. Let $\tilde t=d t$, denoting $\la=1/d$, and $\tau=d r$, and dropping the tilde sign, model \eqref{patc} can be transformed as the following equivalent model:
\begin{equation}\label{ndif}
\begin{cases}
\displaystyle\frac{d u_j}{dt}=\ds\sum_{k=1}^n d_{jk} u_k+\la u_j(m_j-u_j(t-\tau)),&t>0,\;\;j=1,\dots,n,\\
\bm u(t)=\bm \psi(t)\ge \bm 0,&t\in [-\tau,0].
\end{cases}
\end{equation}

When $\tau=0$, the effect of the dispersal on total biomass for \eqref{patc} was studied extensively, see \cite{Arditi2015,Arditi2018,ZhangB2016,ZhangB2015}.
For $\tau>0$, delay induced Hopf bifurcation of model \eqref{patc} (or equivalently, \eqref{ndif})
was investigated for $n=2$. For this case, the characteristic equation has two transcendental terms, and the conditions that
yield purely imaginary roots to the characteristic equation could be obtained thoroughly, see \cite{LiaoLou2014}. However, this method cannot be
applied to the case of $n>3$, where more transcendental terms exist. In this paper, we will consider a general case of $n\ge2$ for model \eqref{patc} (or equivalently, \eqref{ndif}), and the method is also motivated by \cite{Busenberg}. We would like to point out that there are also many results
on stability and bifurcations for other patch models, see \cite{Allenpatch,Bichara2018,Gao2019,KangY2017,LiHC2019,TianC2019,WangXY2016,ZhuHP2018} and references therein.

Throughout the paper, we use the following notations. For $n\ge2$,
\begin{equation}
\begin{split}
&\mathbb R^n=\{\bm u=(u_1,\dots, u_n)^T:u_k\in\mathbb R \;\;\text{for any} \;\;k=1,\dots,n\},\\
&\mathbb R_+^n=\{\bm u=(u_1,\dots, u_n)^T:u_k\ge0 \;\;\text{for any} \;\;k=1,\dots,n\},\\
&\mathbb C^n=\{\bm u=(u_1,\dots, u_n)^T:u_k\in\mathbb C \;\;\text{for any} \;\;k=1,\dots,n\}.\\
\end{split}
\end{equation}
Let $A$ be an $n\times n$ real-valued matrix. We denote the spectral bound of $A$ by
$$s(A):=\max\{\mathcal Re \mu:\mu \;\;\text{is an eigenvalue of}\;\;A\}. $$
Let $\bm u=(u_1,\dots, u_n)^T$ and $\bm v=(v_1,\dots, v_n)^T$ be two vectors.
We write $\bm u\ge \bm v$ if $u_i\ge v_i$ for any $i=1,\dots,n$;
 $\bm u>\bm v$ if $u_i\ge v_i$ for any $i=1,\dots,n$, and there exists
$i_0$ such that $u_{i_0}>v_{i_0}$;
and $\bm u\gg \bm v$ if $u_i>v_i$ for any $i=1,\dots,n$.
Let the complexification of a linear space $Z$ be $Z_\mathbb{C}:= Z\oplus
iZ=\{x_1+ix_2|~x_1,x_2\in Z\}$, and denote the domain
of a linear operator $T$ by $\mathscr{D}(T)$, the kernel of $T$ by $\mathscr{N}(T)$, and the range of $T$ by $\mathscr{R}(T)$.
Moreover, the inner product for $\mathbb C^n$ we choose is $\langle\bm u,\bm v \rangle=\sum_{j=1}^n \overline u_jv_j$, and for any $\bm u\in\mathbb C^n$,
\begin{equation}\label{norm}
\|\bm u\|_\infty=\max_{j=1,\dots,n}|u_j|,\;\;\|\bm u\|_2=\left(\sum_{j=1}^n|u_j|^2\right)^{1/2}.
\end{equation}

The rest of the paper is organized as follows. In Section 2, we consider the existence of Hopf bifurcation when the dispersal rate $d$ is large or near some critical value, and show the asymptotic profile of the Hopf bifurcation value when $d$ tends to infinity or some critical value. In Section
3, we consider the case that $d$ is small, and show the asymptotic profile of the Hopf bifurcation value when $d$ tends to zero.
In Section 4,
some numerical simulations are given to illustrate our theoretical results.

\section{Large (not small) dispersal rate}
In this section, we will consider the existence of Hopf bifurcation of model \eqref{patc} when the dispersal rate $d$ is large or near some critical value. The case of small dispersal rate
will be considered in the next section. Actually, we firstly considered the equivalent system \eqref{ndif}, which is relatively simpler to analyze than the original model \eqref{patc} for this case. Then the results of model \eqref{patc} could be obtained directly from that of model \eqref{ndif}.
\subsection{Some preliminaries}
In this subsection, we mainly give some preliminaries on the properties of the spectrum bound $s(\la Q+D)$, where $Q=\text{diag}(m_i)$, and the asymptotic profile of the steady states of model \eqref{ndif}.
It follows directly from the Perron-Frobenius theorem that:
 $s(D)=0$ is a simple
eigenvalue of $D$ with an eigenvector $\bm \eta\gg\bm 0$, where
\begin{equation}\label{alp1}
\bm {\eta}=(\eta_1,\dots,\eta_n)^T, \;\;\text{and}\;\;\eta_j=\displaystyle\frac{1}{n}\;\;\text{for any}\;\; j=1,\dots,n.
\end{equation}
 Moreover, there exists no other eigenvalue with a nonnegative eigenvector.
Then we have the following result.

\begin{lemma} \label{lm2.1}
Suppose that assumptions ($A_1$) and $(A_2)$ are satisfied. Let $Q={\rm diag}(m_j)$ be a diagonal matrix.
Then the following two statements hold.
\begin{enumerate}
\item [$(i)$] If
\begin{equation}\label{del}
\delta:=\ds\sum_{j=1}^n m_j\ge0,
\end{equation}
then $s\left(\la Q+D\right)>0$ for any $\la>0$.
\item[$(ii)$]If $\delta<0$, then there exists $\la_*>0$ such that
$s\left(\la Q+D\right)<0$ for $\la<\la_*$, $s\left(\la Q+D\right)>0$ for $\la>\la_*$, and
$s\left(\la Q+D\right)=0$ for $\la=\la_*$. Moreover, $s'(\la_*)>0$.
\end{enumerate}
\end{lemma}
\begin{proof}
Since
$\la Q+D$ is quasi-positive and irreducible, we see that $s(\la)$ is a simple
eigenvalue of $\la Q+D$ with an
eigenvector $\bm w=(w_1,\dots,w_n)^T\gg\bm 0$. Without loss of generality, we assume that
\begin{equation}
\label{sumw}
\sum_{j=1}^n w_j=1\;\;\text{for any}\;\;\la>0.
\end{equation}
 Then
\begin{equation}\label{p1}
s(\la)w_j=\la m_jw_j+\ds\sum_{k\ne j}d_{jk}w_k-\ds\sum_{k\ne j}d_{jk}w_j\;\;\text{for any}\;\;j=1,\dots,n.
\end{equation}
Differentiating \eqref{p1} with respect to $\la$, we have
\begin{equation}\label{p2}
s'w_j+sw_j'=m_jw_j+\la m_jw_j'+\ds\sum_{k\ne j}d_{jk}w_k'-\ds\sum_{k\ne j}d_{jk}w_j'\;\;\text{for any}\;\;j=1,\dots,n.
\end{equation}
Multiplying \eqref{p2} by $w_j$ yields
\begin{equation}\label{p3}
s'w^2_j+sw_j'w_j=m_jw^2_j+\la m_jw_j'w_j+\ds\sum_{k\ne j}d_{jk}w_k'w_j-\ds\sum_{k\ne j}d_{jk}w_j'w_j\;\;\text{for any}\;\;j=1,\dots,n.
\end{equation}
Plugging \eqref{p1} into \eqref{p3}, we have
\begin{equation}\label{p4}
s'w_j^2=m_jw_j^2+\sum_{k\ne j} d_{jk}w_jw_k'-\sum_{k\ne j} d_{jk}w_j'w_k.
\end{equation}
Summing \eqref{p4} over all $j$, and noticing that $(d_{jk})_{n\times n}$ is symmetric, we have
 \begin{equation}\label{p6}
 s'\sum_{j=1}^n w_j^2=\ds\sum_{j=1}^n m_jw_j^2+\ds\sum_{j=1}^n\sum_{k\ne j} d_{jk} w_jw_k'-\ds\sum_{j=1}^n\ds\sum_{k\ne j}d_{jk}w_j'w_k
 =\ds\sum_{j=1}^n m_j w_j^2.
\end{equation}
Clearly, when $\la=0$, $s(\la)=0$ and $w_j=1/n$ for $j=1,\dots,n$, and consequently
\begin{equation}\label{fide}
s'(0)=\displaystyle\frac{\delta}{n}=\displaystyle\frac{1}{n}\sum_{j=1}^n m_j.
\end{equation}

To show the properties of function $s(\la)$, we need to analyze the second derivative of $s(\la)$.
Then differentiating \eqref{p2} with respect to $\la$, we have
\begin{equation}\label{sp1}
s''w_j+2s'w_j'+sw_j''=2m_jw_j'+\la m_jw_j''+\sum_{k\ne j}d_{jk}w_k''-\sum_{k\ne j}d_{jk} w_j''.
\end{equation}
Multiplying \eqref{sp1} by $w_j$ gives
\begin{equation}\label{sp2}
s''w^2_j+2s'w_j'w_j+sw_j''w_j=2m_jw_j'w_j+\la m_jw_j''w_j+\sum_{k\ne j}d_{jk}w_k''w_j-\sum_{k\ne j}d_{jk} w_j''w_j.
\end{equation}
Plugging \eqref{p1} and \eqref{p4} into \eqref{sp2},
we have
\begin{equation}\label{sp3}
s''w_j^2=\sum_{k\ne j}d_{jk}\left(w_jw_k''-w_j''w_k\right)-2\sum_{k\ne j}d_{jk}\left[w'_jw'_k-(w'_j)^2\f{w_k}{w_j}\right].
\end{equation}
Summing \eqref{sp3} over all $j$ and noticing that $(d_{jk})_{n\times n}$ is symmetric, we have
\begin{equation}\label{sedi}
\begin{split}
s''\ds\sum_{j=1}^nw_j^2=&\ds\sum_{j=1}^n\sum_{k\ne j}d_{jk}\left(w_jw_k''-w_j''w_k\right)-2\sum_{j=1}^n\sum_{k\ne j}d_{jk}\left[w'_jw'_k-(w'_j)^2\f{w_k}{w_j}\right]\\
=&-2\sum_{j=1}^n\sum_{k\ne j}d_{jk}\left[w'_jw'_k-(w'_j)^2\f{w_k}{w_j}\right]
=-2\sum_{j=1}^n\sum_{k\ne j}d_{jk}\left[w'_jw'_k-(w'_k)^2\f{w_j}{w_k}\right]\\
=&\sum_{j=1}^n\sum_{k\ne j}d_{jk}\left(\displaystyle\frac{\sqrt{w_k}}{\sqrt{w_j}}w_j'-\displaystyle\frac{\sqrt{w_j}}{\sqrt{w_k}}w_k'\right)^2\ge0.\\
\end{split}
\end{equation}
 Then it follows from \eqref{sedi} that
$s''(\la)\ge0$ for any $\la>0$. Moreover, $s''(\la)=0$ if and only if
\begin{equation*}
\displaystyle\frac{w_j'}{w_j}=\displaystyle\frac{w_k'}{w_k}\;\;\text{for any}\;\;j\ne k,
\end{equation*}
which implies that $w_j'=cw_j$ for any $j=1,\dots,n$. It follows from \eqref{sumw} that
$\sum_{j=1}^nw_j'=0$ for any $\la>0$. Then $c=0$, and consequently $w_j'=0$ for any $\la>0$ and $j=1,\dots,n$.
This, combined with \eqref{p2}, implies that $$m_j=s'(\la)\;\;\text{for any}\;\;j=1,\dots,n.$$
Therefore $s''(\la)\ge0$ and the equality holds if and only if $m_1=m_2=\dots=m_n$.

If $\delta=\ds\sum_{j=1}^n m_j>0$, then $s'(0)>0$ from \eqref{fide}. This, combined with the fact that $s''(\la)\ge0$, implies that $s'(\la)>0$ for any $\la>0$. If $\delta=0$, we see from assumption $(A_2)$ that $s''(\la)>0$, which also implies that $s'(\la)>0$ for any $\la>0$. Therefore, if $\delta\ge0$, $s(\la)$ is strictly monotone increasing for $\la>0$,
and consequently $s(\la Q+D)=s(\la)>0$ for any $\la>0$. This completes the proof of $(i)$.

If $\delta<0$, then
$s'(0)<0$, and $s''(\la)>0$ for any $\la>0$ from assumption $(A_2)$.
It follows from \cite[Lemma 3.4]{Allenpatch} that
\begin{equation}\label{lainf}
\lim_{\la\to\infty}s\left(Q+\displaystyle\frac{1}{\la} D\right)=\ds\max_{j=1,\dots,n}\{m_j\}.
\end{equation}
Noticing that
$$\displaystyle\frac{s(\la)}{\la}=s\left(Q+\displaystyle\frac{1}{\la} D\right),$$
we see from \eqref{lainf} that
\begin{equation*}
\ds\lim_{\la\to\infty}\displaystyle\frac{s(\la)}{\la}=\ds\max_{j=1,\dots,n}\{m_j\}>0.
\end{equation*}
Then $s(\la)>0$ for sufficiently large $\la$.
This, combined with the fact that $s'(0)<0$ and $s''(\la)>0$, implies that
$(ii)$ holds.
\end{proof}

Then we consider the global dynamics of model \eqref{ndif} when $\tau=0$,
which has been investigated in \cite{LiShuai2010}. Here we include it for completeness, and mainly show the effect of parameter $\la$ on the global dynamics.
\begin{lemma}\label{lm2.2}
Suppose that assumptions ($A_1$) and $(A_2)$ hold, and $\tau=0$.
\begin{enumerate}
\item[$(i)$] If $\delta\ge0$, where $\delta$ is defined as in \eqref{del},
 then system \eqref{ndif} admits a unique positive equilibrium
$\bm u_\la=(u_{\la, 1},\dots, u_{\la, n})^T\gg\bm 0$, which is globally asymptotically stable.

\item [$(ii)$] If $\delta<0$, then the trivial equilibrium $\bm 0= (0,\dots,0)^T$ of \eqref{ndif} is globally asymptotically for $\la\in(0,\la_*]$, and for $\la>\la_*$, system \eqref{ndif} admits a unique positive equilibrium
$\bm u_\la=(u_{\la,1},\dots, u_{\la,n})^T\gg\bm 0$, which is globally asymptotically stable, where $\la_*$ is defined as in Lemma \ref{lm2.1}.
\end{enumerate}
\end{lemma}
\begin{proof}
Let $\Psi_t$ be the corresponding flow of \eqref{ndif} for $\tau=0$, and let $$\bm{ g}(\bm u)=\left(g_1(\bm u),\dots,\ g_n(\bm u)\right)^T$$ be the vector field of \eqref{ndif}, where
\begin{equation*}
g_j(\bm u)=\sum_{k=1}^n d_{jk} u_k+\la u_j(m_j-u_j).
\end{equation*}
Note that $L$ is irreducible and quasi-positive from $(A_1)$. Then it follows from \cite[Theorem B.3]{SmithW} that
$\Psi_t$ is strongly positive and monotone.
For any give $\alpha\in(0,1)$ and $\bm u\gg \bm 0$,
$$g_j(\alpha\bm u)-\alpha g_j(\bm u)=\alpha \la u^2_j(1-\alpha)>0\;\;\text{for any}\;\;j=1,\dots,n,$$
which implies that $\bm g(\bm u)$ is strictly sublinear.
Note that the solution of \eqref{ndif} is bounded. Then we see from \cite[Corollary 3.2]{ZhaoJing}
 that the global dynamics of model \eqref{ndif} for $\tau=0$ is determined by $s(\la Q+D)$, where $Q=\text{diag}(m_i)$. That is, for $\tau=0$,
 \begin{enumerate}
 \item [$(i)$]
if $s(\la Q+D)\le0$, then the trivial equilibrium $\bm 0=(0,\dots,0)^T$ is globally asymptotically stable;
\item [$(ii)$]
if $s(\la Q+D)>0$, then system \eqref{ndif} admits a unique positive equilibrium $\bm u_\la=(u_{\la,1},\dots, u_{\la_n,n})^T\gg\bm 0$, which is globally asymptotically stable.
\end{enumerate}
 This, combined with Lemma \ref{lm2.1} implies that, for $\tau=0$,
\begin{enumerate}
\item [$(i)$]
if $\delta\ge0$,
 then system \eqref{ndif} admits a unique positive equilibrium
$\bm u_\la$, which is globally asymptotically stable;
\item [$(ii)$]
if $\delta<0$, then the trivial equilibrium $\bm 0=(0,\dots,0)^T$ of \eqref{ndif} is globally asymptotically for $\la\in(0,\la_*]$, and for $\la>\la_*$, system \eqref{ndif} admits a unique positive equilibrium
$\bm u_\la$, which is globally asymptotically stable.
\end{enumerate}
\end{proof}

To show the existence of Hopf bifurcation for model \eqref{ndif}, we need to show the asymptotic profile of $\bm u_\la$. Clearly,  $\bm u_\la$ satisfies
\begin{equation}\label{ndifsteady}
\ds\sum_{k=1}^n d_{jk} u_k+\la u_j(m_j-u_j)=0,\;\;j=1,\dots,n.\\
\end{equation}
Then we have the following result.
\begin{lemma}\label{lm2.3}
Suppose that assumptions ($A_1$) and $(A_2)$ hold. Let $\bm u_\la=(u_{\la,1},\dots,u_{\la,n})^T$ be the positive equilibrium of \eqref{ndif} (or respectively, positive solution of \eqref{ndifsteady}).
\begin{enumerate}
\item[$(i)$] If $\delta>0$, then $\bm u_\la$ is continuously differentiable for $\la\in(0,\infty)$. Moreover, define
\begin{equation}\label{lims1}
\bm u_{0}=(u_{0,1},\dots,u_{0,n})^T, \;\;\text{where}\;\; u_{0,j}=\displaystyle\frac{\delta}{n}\;\;\text{for}\;\;j=1,\dots,n,
\end{equation}
and then $\bm u_\la$ is continuously differentiable for $\la\in[0,\infty)$.
\item [$(ii)$] If $\delta<0$, then there exist $\alpha_\la$ and ${\bm \xi}_\la$, which are continuously differentiable for $\la\in[\la_*,\infty)$, such that
${\bm u}_\la$ has the following form
\begin{equation}\label{formu}
{\bm u}_\la=\alpha_\la(\la-\la_*)\left[\hat{\bm\eta}+(\la-\la_*){\bm \xi}_\la\right]\;\; \text{for}\;\;\la\in(\la_*,\infty),
\end{equation}
where $\alpha_\la>0$, $\hat{\bm \eta}= (\hat\eta_1,\dots,\hat\eta_n)^T\gg\bm 0$ ($\sum_{j=1}^n \hat\eta_j=1$) is the eigenvector of $\la_*Q+D$ associated with eigenvalue $s(\la_*Q+D)=0$, and
\begin{equation}\label{X1}
 \bm \xi_\la=(\xi_{\la,1},\dots,\xi_{\la,n})^T\in \hat X_1:=\left\{(x_1,\dots,x_n)^T\in \mathbb R^n:\sum_{j=1}^n\hat \eta_jx_j=0\right\}.
\end{equation}
Moreover,
\begin{equation}\label{lims2}
 \alpha_{\la_*}=\displaystyle\frac{\sum_{j=1}^nm_j\hat\eta_j^2}{\la_*\sum_{j=1}^n\hat\eta_j^3},
\end{equation}
and $\bm \xi_{\la_*}=(\xi_{\la_*,1},\dots,\xi_{\la_*,n})^T$ is the unique solution of the following equation
\begin{equation}\label{xi}
\ds\sum_{k=1}^nd_{jk}\xi_k+\la_*m_j\xi_j+m_j\hat\eta_j-\la_*\alpha_{\la_*}\hat\eta_j^2=0, \;\;j=1,\dots,n.
\end{equation}
\end{enumerate}
\end{lemma}
\begin{proof}
We first prove $(i)$. Clearly, $\bm u_\la$ is continuously differentiable for $\la\in(0,\infty)$.
Then we will show that $\bm u_\la$ is continuously differentiable for $\la\in[0,\infty)$, if $\bm u_0$ is defined as in \eqref{lims1}.

Let
\begin{equation}\label{hX1}
X_1=\left\{(x_1,\dots, x_n)^T\in\mathbb R^n:\sum_{i=1}^n x_i=0\right\}.
\end{equation}
Then $\mathbb R^n={\rm span}\{\bm\eta\}\oplus X_1$.
Letting
\begin{equation}\label{spe}
\bm u= c\bm \eta +\bm w,
 \end{equation}
where $c\in\mathbb R$, $\bm\eta$ is defined as in \eqref{alp1}, and $\bm w\in X_1$,
and substituting it
into \eqref{ndifsteady}, we see that
\begin{equation*}
h_j(\la,c,\bm w):=\sum_{k=1}^n d_{jk}w_k+\la(c\eta_j+w_j)\left(m_j-c\eta_j-w_j\right)=0\;\;\text{for}\;\;j=1,\dots,n.
\end{equation*}
Define $\bm h: \mathbb R\times \mathbb R\times X_1 \mapsto \mathbb R^n$ by
$$\bm h(\la,c,\bm w)=\left(h_1(\la,c,\bm w),\dots, h_n(\la,c,\bm w)\right)^T,$$
where $\bm w=(w_1,\dots,w_n)^T\in X_1$.
Then
 we see that $(\la, \bm u)$ solves \eqref{ndifsteady}, where $\la>0$ and $\bm u\in \mathbb R^n $,
if and only if $\bm h(\la,c,\bm w)=\bm 0$
is solvable for some value of $\la>0$, $c\in \mathbb R$ and $\bm w\in X_1$.

Clearly, $\bm h(0,c,\bm 0)=\bm 0$ for any $c\in \mathbb R$. A direct computation implies that
\begin{equation*}
\begin{split}
D_{(\la,\bm w)}\bm h(0,c,\bm 0)[\sigma,\bm v]=\left(\begin{array}{c}
\sum_{k=1}^nd_{1k}v_k +\displaystyle\frac{\sigma c}{n}\left(m_1-\displaystyle\frac{c}{n}\right)\\
\sum_{k=1}^nd_{2k}v_k + \displaystyle\frac{\sigma c}{n}\left(m_2-\displaystyle\frac{c}{n}\right)\\
\vdots\\
\sum_{k=1}^nd_{nk}v_k +\displaystyle\frac{\sigma c}{n}\left(m_n-\displaystyle\frac{c}{n} \right)\\
\end{array}\right),
\end{split}
\end{equation*}
where $\bm v =(v_1,\dots,v_n)^T$, and
$D_{(\la,\bm w)}\bm h(0,c,\bm 0)$
is the Fr\'echet derivative of $\bm h(\la,c,\bm w)$ with respect to $(\la,\bm w)$ at $(0,c,\bm 0)$.
Since
$$\left(\displaystyle\frac{\sigma \delta}{n}\left(m_1-\displaystyle\frac{\delta}{n}\right),\dots, \displaystyle\frac{\sigma \delta}{n}\left(m_n-\displaystyle\frac{\delta}{n}\right)\right)^T\in X_1,$$
where $\delta=\sum_{j=1}^n m_j$ is defined as in \eqref{del},
it follows that there exists a unique $\bm v^*\in X_1$ such that
$$D_{(\la,\bm w)}\bm h(0,\delta,\bm 0)[1,\bm v^*]=\bm 0,$$
and consequently,
$$\mathscr{N}( D_{(\la,\bm w)}\bm h(0,\delta,\bm 0))=\{(s,s\bm v^*):s\in\mathbb R\}.$$
Then we calculate that
$$D_cD_{(\la,\bm w)}\bm h(0,\delta,\bm 0)[1,\bm v^*]=\left(\displaystyle\frac{m_1}{n}-\displaystyle\frac{2\delta}{n^2},\dots,\displaystyle\frac{m_n}{n}-\displaystyle\frac{2\delta}{n^2}\right)^T,$$
where $D_cD_{(\la,\bm w)}\bm h(0,\delta,\bm 0)$ is the Fr\'echet derivative of $D_{(\la,\bm w)}\bm h(\la,c,\bm w)$ with respect to $c$ at $(0,\delta,\bm 0)$.
We claim that
\begin{equation}\label{cl1}D_cD_{(\la,\bm w)}\bm h(0,\delta,\bm 0)[1,\bm v^*]\not\in \mathscr{R}\left(D_{(\la,\bm w)}\bm h(0,\delta,\bm 0)\right).
\end{equation}
If it is not true, then there exists $(\hat\sigma,\hat {\bm v})$ such that
\begin{equation*}
\begin{split}
\sum_{k=1}^nd_{jk}\hat v_k + \displaystyle\frac{\hat \sigma \delta}{n}\left(m_j-\displaystyle\frac{\delta}{n}\right)=\displaystyle\frac{m_j}{n}-\displaystyle\frac{2\delta}{n^2}\;\;\text{for any}\;\;j=1,\dots,n.
\end{split}
\end{equation*}
This implies that
$$\left(\displaystyle\frac{m_1}{n}-\displaystyle\frac{2\delta}{n^2},\dots,\displaystyle\frac{m_n}{n}-\displaystyle\frac{2\delta}{n^2}\right)^T\in X_1,$$
which contradicts the fact that
$$\sum_{j=1}^n m_j-2\delta=-\delta\ne0.$$
Therefore, Eq. \eqref{cl1} holds, and we see from the Crandall-Rabinowitz bifurcation theorem
\cite{Crandall} that the solutions of $\bm h(\la,c,\bm w)=0$ near $(0,\delta,\bm 0)$ consist precisely by the curves
$\{(c,0,0):c\in \mathbb {R}\}$ and $$\{(\la(y),c(y),\bm w(y)): y\in(-\epsilon,\epsilon)\},$$ where $(\la(y),c(y),\bm w(y))$ is continuously differentiable,  $\la(0)=0$, $c(0)=\delta$, $\bm w(0)=\bm 0$, $\bm w'(0)=\bm v^*$, and $\la'(0)=1$.
Since $\la'(0)=1>0$, $\la(y)$ has a continuously differentiable inverse function $y(\la)$ for small $y$. Then there exist $\la_1>0$ and
a continuously  differentiable function $$\hat {\bm u}_\la=c(y(\la))\bm \eta+\bm w(y(\la)):[0,\la_1]\mapsto \mathbb R^n$$ such
that $\hat {\bm u}_\la$ satisfies \eqref{ndifsteady}
for $\la\in(0,\la_1]$.
Moreover, $$\hat{\bm  u}_0=c(y(0))\bm \eta+\bm w(y(0))=c(0)\bm \eta+\bm w(0)=\delta\bm\eta.$$
The uniqueness of the positive equilibrium of \eqref{ndif} implies that $\bm u_\la=\hat {\bm u}_\la$ for $\la\in(0,\la_1]$.
Therefore, letting $\bm u_0=(u_{0,1},\dots,u_{0,n})^T$,
we obtain that $\bm u_\la$ is continuously differentiable for $\la\in[0,\infty)$.

Now we prove $(ii)$. Since
\begin{equation*}
\mathbb R^n={\rm span}\{\hat {\bm\eta}\}\oplus \hat X_1,
\end{equation*}
it follows that $\bm u_\la$ can be represented as \eqref{formu}, and $\alpha_\la$ and $\bm \xi_\la$ are continuously differentiable for $\la\in(\la_*,\infty)$.
Then we will show that $\alpha_\la$ and $\bm \xi_\la$ are continuously differentiable for $\la\in[\la_*,\infty)$, if $\alpha_{\la_*}$ and $\bm \xi_{\la_*}$ are defined as in \eqref{lims2} and \eqref{xi} respectively.

It follows from Lemma \ref{lm2.1} that $s'(\la_*)>0$. Then we see from \eqref{p6} that
\begin{equation*}
\sum_{j=1}^nm_j\hat \eta_j^2>0,
\end{equation*}
which implies that $\alpha_{\la_*}$ is well defined and positive. Since
\begin{equation*}
\sum_{j=1}^nm_j\hat \eta_j^2-\la_*\alpha_{\la_*}\sum_{j=1}^n\hat\eta_j^3=0,
\end{equation*}
we see that ${\bm \xi}_{\la_*}$ is also well defined.

Substituting
$$\bm u=\alpha(\la-\la_*)\left[\hat{\bm \eta}+(\la-\la_*){\bm \xi}\right]$$
into \eqref{ndifsteady}, where $\bm \xi=(\xi_1,\dots, \xi_n)^T\in
\hat X_1$,  we have, for any $j=1,\dots,n$,
\begin{equation*}
\begin{split}
&p_j(\la,\alpha,\bm\xi)
:=\sum_{k=1}^nd_{jk}\xi_k+\la_*m_j\xi_j+m_j\left[\hat\eta_j+(\la-\la_*)\xi_j\right]-\la\alpha\left[\hat\eta_j+(\la-\la_*)\xi_j\right]^2=0.
\end{split}
\end{equation*}
Define $\bm p(\la, \alpha,\bm \xi):\mathbb R\times \mathbb R\times \hat X_1\mapsto \mathbb R^n$ by
\begin{equation*}
\bm p(\la, \alpha,\bm \xi)=\left(p_1(\la,\alpha,\bm\xi),\dots,p_n(\la,\alpha,\bm\xi)\right)^T.
\end{equation*}
Then $(\la,\bm u)$ solves \eqref{ndifsteady} if and only if $\bm p(\la,\alpha,\bm \xi)=\bm 0$.
Clearly, $\bm p(\la_*,\alpha_{\la_*},{\bm \xi}_{\la_*})=\bm 0$, and  the Fr\'echet derivative of $\bm p$ with respect to $(\alpha,\bm \xi)$ at $(\la_*,\alpha_{\la_*},{\bm \xi}_{\la_*})$ is
\begin{equation*}
D_{(\alpha,\bm \xi)}\bm p(\la_*,\alpha_{\la_*},{\bm \xi}_{\la_*})[\epsilon,\bm v]=\left(\begin{array}{c}
\sum_{k=1}^nd_{1k}v_k+\la_*m_1v_1-\la_*\hat\eta_1^2\epsilon \\
\sum_{k=1}^nd_{2k}v_k+\la_*m_2v_2-\la_*\hat\eta_2^2\epsilon \\
\vdots\\
\sum_{k=1}^nd_{nk}v_k+\la_*m_nv_n-\la_*\hat\eta_n^2\epsilon \\
\end{array}\right),
\end{equation*}
where $\bm v =(v_1,\dots,v_n)^T\in\hat X_1$. A direct computation implies that
$D_{(\alpha,\bm \xi)}\bm p(\la_*,\alpha_{\la_*},{\bm \xi}_{\la_*})$ is bijective from $\mathbb R\times\hat X_1$ to $\mathbb R^n$.
Then from the implicit function theorem, there
exist $\la^*>\la_*$ and a continuously differentiable mapping $\la\in[\la_*,\la^*]\mapsto(\tilde \alpha_\la,\tilde{\bm \xi}_\la)\in {\mathbb R}\times \hat X_1$ such that
${\bm p}(\la,\tilde \alpha_\la,\tilde{\bm \xi}_\la)=\bm 0$, and $\tilde \alpha_\la=\alpha_{\la_*}$ and $\tilde{\bm \xi}_{\la_*}={\bm \xi}_{\la_*}$ for $\la=\la_*$.
The uniqueness of the positive equilibrium of \eqref{ndif} implies that $\alpha_\la=\tilde \alpha_\la$ and $\bm \xi_\la=\tilde{\bm \xi}_\la$ for $\la\in(\la_*,\la^*]$. Therefore, $\alpha_\la$ and $\bm \xi_\la$ are continuously differentiable
for $\la\in[\la_*,\infty)$.
\end{proof}

\subsection{The eigenvalue problem}
In this section,
we consider the eigenvalue problem associated with the positive equilibrium $\bm u_\la=(u_{\la,1},\dots,u_{\la,n})^T$.
 Linearizing \eqref{ndif} at $\bm{u}_\la$, we have
\begin{equation}
\label{linear}
\displaystyle\frac{d \bm {v}}{d t} =D\bm{ v}+\la {\rm diag}\left(m_j-u_{\la,j}\right)\bm v-\la {\rm diag}( u_{\la,j})\bm{ v}(t-\tau).
\end{equation}
It follows from \cite{Hale1971} that the infinitesimal generator $A_\tau(\la)$ of the solution semigroup of
\eqref{linear} is defined by
\begin{equation}\label{Ataula}A_\tau(\la) \bm{\Psi}=\dot{\bm{\Psi}},\end{equation}
and the domain of $A_\tau(\la)$ is
\begin{equation*}
\begin{split}
 &\mathscr{D}(A_\tau(\la)) = \big\{\bm {\Psi}\in C_\mathbb{C}
\cap C^1_\mathbb{C}:\ \bm \Psi(0)\in \mathbb{C}^n,\dot{\bm \Psi}(0)=D\bm{\Psi}(0)+\la \textrm{diag}\left(m_j-u_{\la,j}\right)\bm{ \Psi}(0)\\
&~~~~~~~~~~~~~~~~~-\la\textrm{ diag}( u_{\la,j})\bm{ \Psi}(-\tau) \big\},
\end{split}
\end{equation*}
where $C_{\mathbb C}=C([-\tau,0],\mathbb{C}^n)$ and $C^1_\mathbb{C}=C^1([-\tau,0],\mathbb{C}^n)$. Then, $\mu\in\mathbb{C}$ is an eigenvalue of $A_\tau(\la)$ iff there exists $\bm{\psi}=(\psi_1,\dots,\psi_n)^T(\ne\bm 0)\in\mathbb{C}^n$ such that $\Delta(\la,\mu,\tau)\bm \psi=\bm 0$,
where
\begin{equation}\label{triangle}
\begin{split}
&\Delta(\la,\mu,\tau)\bm{\psi}
:=D\bm{\psi}+\la \textrm{diag}\left(m_j-u_{\la,j}\right)\bm \psi-\la e^{-\mu\tau}\textrm{diag}( u_{\la,j})\bm{ \psi}-\mu\bm{\psi}.
\end{split}
\end{equation}
Therefore, $A_\tau(\la)$ has a purely imaginary eigenvalue $\mu=\textrm i\nu\
(\nu>0)$ for some $\tau\ge0$, iff
\begin{equation}\label{eigen}
\begin{split}
D\bm{\psi}+\la \textrm{diag}\left(m_j-u_{\la,j}\right)\bm \psi-\la e^{-\textrm i\theta}\textrm{diag}( u_{\la,j})\bm{ \psi}-\textrm{i}\nu\bm{\psi}=\bm 0
\end{split}
\end{equation}
is solvable for some value of $\nu>0$, $\theta\in[0,2\pi)$, and $\bm{\psi}(\ne \bm 0)\in \mathbb{C}^n$.
Note from Lemma \ref{lm2.3} that the properties of $\bm u_\la$ are different for $\delta>0$ and $\delta<0$, where $\delta=\sum_{j=1}^n m_j$ is defined as in Eq. \eqref{del}.
Then the following discussion is divided into two cases:
\begin{equation*}
\text{Case I:}\;\;\delta>0,\;\;\;\;\text{Case II:}\;\;\delta<0.
\end{equation*}

\subsection{Case I: $\delta>0$}
In this subsection, we consider the stability/instability of the unique positive equilibrium $\bm u_\la$ of system \eqref{ndif} and the associated Hopf bifurcation near $\bm u_\la$ for the case of $\delta>0$.
It follows from Lemma \ref{lm2.3} that $\bm u_\la$ is continuously differentiable for $\la\in[0,\infty)$, and $\bm u_0=(u_{0,1},\dots,u_{0,n})^T$, where
\begin{equation*}
u_{0,j}=\displaystyle\frac{\delta}{n}\;\;\text{for}\;\;j=1,\dots,n.
\end{equation*}
We firstly give an estimate for solutions of Eq. \eqref{eigen}.
\begin{lemma}\label{nu}
Assume that $(\nu_\la,\theta_\la,\bm{\psi}_\la )$ solves \eqref{eigen}, where $\bm{\psi}_\la=(\psi_{\la,1},\dots,\psi_{\la,n})^T(\ne\bm 0)
\in\mathbb{C}^n$.  Then for any $\la_1>0$,
$\left|\ds
\frac{\nu_\la}{\la}\right|$ is bounded for
$\la\in(0,\la_1]$.\end{lemma}
\begin{proof}
Substituting $(\nu_\la,\theta_\la,\bm{\psi}_\la )$ into \eqref{eigen} and multiplying it by $\left(\overline {\psi}_{\la,1},\dots,\overline
{\psi}_{\la,n}\right)$, we have
\begin{equation*}
\left(\overline {\psi}_{\la,1},\dots,\overline
{\psi}_{\la,n}\right)\left[D{\bm \psi}_\la +\la \textrm{diag}\left(m_j-u_{\la,j}\right){\bm \psi}_\la-\la e^{-\textrm i\theta_\la}\textrm{diag}( u_{\la,j}){\bm \psi}_\la-\textrm{i}\nu_\la\bm{\psi}_\la\right]=0,
\end{equation*}
and consequently,
\begin{equation*}
\begin{split}
&\sum_{j=1}^n\sum_{k\ne j}d_{jk}\left(\overline {\psi}_{\la,j}\psi_{\la,k}-\left|\psi_{\la,j}\right|^2\right)
+\la\sum_{j=1}^n\left(m_j-u_{\la,j}\right)\left|\psi_{\la,j}\right|^2\\
&-\la e^{-\textrm{i}\theta_\la }\sum_{j=1}^nu_{\la,j}\left|\psi_{\la,j}\right|^2-\textrm{i}\nu_\la\sum_{j=1}^n\left|\psi_{\la,j}\right|^2=0.
\end{split}
\end{equation*}
Since $\left(d_{jk}\right)_{n\times n}$ is symmetric, it follows that
\begin{equation*}
\begin{split}
\sum_{j=1}^n\sum_{k\ne j}d_{jk}\left(\overline {\psi}_{\la,j}\psi_{\la,k}-\left|\psi_{\la,j}\right|^2\right)
=&\sum_{j=1}^n\sum_{k\ne j}d_{jk}\left( \psi_{\la,j}\overline{\psi}_{\la,k}-\left|\psi_{\la,k}\right|^2\right)\\
=&\displaystyle\frac{1}{2}\sum_{j=1}^n\sum_{k\ne j}d_{jk}\left[-\left|\psi_{\la,j}\right|^2+  \left(\psi_{\la,k}\overline{\psi}_{\la,j}+\psi_{\la,j}\overline {\psi}_{\la,k}\right)-\left|\psi_{\la,k}\right|^2\right]\\
\le&-\displaystyle\frac{1}{2}\sum_{j=1}^n\sum_{k\ne j}d_{jk}\left(\left|\psi_{\la,j}\right|-\left|\psi_{\la,k}\right|\right)^2\le0.
\end{split}
\end{equation*}
Then
\begin{equation*}
\nu_\la\sum_{j=1}^n\left|\psi_{\la,j}\right|^2=\la \sin\theta_\la \sum_{j=1}^nu_{\la,j}\left|\psi_{\la,j}\right|^2,
\end{equation*}
which implies that
\begin{equation*}
\left|\ds
\frac{\nu_\la}{\la}\right|\le\max_{\la\in[0,\la_1]}\|\bm u_\la\|_\infty \;\;\text{for any}\;\;\la\in(0,\la_1),
\end{equation*}
where $\|\cdot\|_\infty$ is defined as in \eqref{norm}. This completes the proof.
\end{proof}

The following result is similar to Lemma 2.3 of \cite{Busenberg} and
we omit the proof here.
\begin{lemma}\label{lem21}
Assume that $\bm z\in (X_1)_{\mathbb C}$, where $X_1$ is defined as in \eqref{hX1}. Then $|
\langle\bm z, D\bm z\rangle|\ge
\gamma_2\|\bm z\|_2^2$, where $\|\cdot \|_2$ is defined as in \eqref{norm}, and $-\gamma_2(<0)$ is the second
largest eigenvalue of matrix $D$.
\end{lemma}

For $\la\in(0,\la_1]$, ignoring a scalar factor, $\bm \psi$ in Eq. \eqref{eigen} can be represented as
\begin{equation}
\label{eigen2}
\begin{split}
&\bm \psi= \beta\bm \eta+\la\bm  z,\;\;\; \bm z\in (X_1)_{\mathbb{C}},\; \; \beta\geq0, \\
 &\|\bm \psi\|^2_{2}=\beta^2 \|\bm \eta\|^2_2
 +\la^2\|\bm z\|^2_{2}=  \|\bm \eta\|^2_2,
 \end{split}
 \end{equation}
 where $\bm \eta$ and $X_1$ are defined as in \eqref{alp1} and \eqref{hX1}, respectively.
 Clearly, \begin{equation}\label{chX1}
(X_1)_{\mathbb C}=\left\{(x_1,\dots, x_n)^T\in\mathbb {C}^n:\sum_{i=1}^n x_i=0\right\}.
\end{equation}
Then, substituting the first equation of \eqref{eigen2} and $\nu=\la h$
into Eq. \eqref{eigen}, we see that $(\nu,\theta,\bm \psi)$ solves \eqref{eigen}, where $\nu>0$, $\theta\in[0,2\pi)$ and $\bm \psi\in \mathbb C^n(\|\bm \psi\|^2_{2}= \|\bm \eta\|_2^2)$,
if and only if the following system:
\begin{equation}\label{g1}
\begin{cases}\bm g_1(\bm z,\beta,h,\theta,\la):=D\bm{z}+ \textrm{diag}\left(m_j-u_{\la,j}\right)(\beta{\bm \eta}+\la\bm z)\\
~~~~~~~~~~~~~~~~~~~~~~~- e^{-\textrm i\theta}\textrm{diag}( u_{\la,j})(\beta{\bm \eta}+\la\bm z)-\textrm{i}h(\beta{\bm \eta}+\la\bm z)=\bm 0,\\
 g_2(\bm z,\beta,\la):=(\beta^2-1)\|\bm \eta\|^2_2+\la^2\|\bm z\|^2_{2}=0
\end{cases}
\end{equation}
has a solution $(\bm z,\beta,h,\theta)$, where $\bm z\in (X_1)_{\mathbb{C}}$, $\beta\ge0$, $h>0$ and $\theta\in[0,2\pi)$.
Define
$\bm G:(X_1)_{\mathbb C}\times \mathbb{R}^4\to
\mathbb C^n\times \mathbb{R}$ by $G=(\bm g_1,g_2)^T$. We firstly show that
$\bm G(\bm z,\beta,h,\theta,\la )=0$ has a unique solution for $\la=0$.
\begin{lemma}\label{l25}
The following equation \begin{equation}\label{3.5G}
\begin{cases}
\bm G(\bm z,\beta,h,\theta,0)=\bm 0\\
\bm z\in (X_1)_{\mathbb{C}},\;h\ge0,\;\beta\ge0,\; \theta\in[0,2\pi]\\
\end{cases}
\end{equation} has a unique solution $(\bm z_{0},\beta_{0},h_{0},\theta_{0})$, where
\begin{equation}\label{lastar}
    \beta_{0}=1,\;\;\theta_{0}=\pi/2,\;\;h_{0}=\displaystyle\frac{\delta}{n},
\end{equation}
and $\bm z_{0}=(z_{0,1},\dots,z_{0,n})^T\in(X_1)_{\mathbb C}$ is the unique solution of
\begin{equation}\label{z0}
\sum_{k=1}^n d_{jk} z_k+\left(m_j-\displaystyle\frac{\delta}{n}\right)\displaystyle\frac{1}{n}=0, \;\;j=1,\dots,n.
\end{equation}
\end{lemma}
\begin{proof}
Clearly, $g_2(\bm z,\beta,0)=0$ if and only if $\beta=\beta_{0}=1$. Noticing that $$\bm u_0=\left(\displaystyle\frac{\delta}{n},\dots, \displaystyle\frac{\delta}{n}\right),$$
where $\delta$ is defined as in Eq. \eqref{del},
and substituting $\beta=\beta_0$ into
$\bm g_1(\bm z,\beta,h,\theta,0)=\bm 0$, we see that
\begin{equation}\label{subequi}
\sum_{k=1}^n d_{jk} z_k+\left(m_j-\displaystyle\frac{\delta}{n}\right)\displaystyle\frac{1}{n}-\displaystyle\frac{\delta}{n^2}e^{-\textrm{i}\theta}-\displaystyle\frac{\textrm {i}h}{n}=0, \;\;j=1,\dots,n.
\end{equation}
Then Eq. \eqref{subequi} has a solution $(\bm z,h,\theta)$, where $\bm z\in(X_1)_{\mathbb C}$, $h\ge0$, $\theta\in[0,2\pi]$, iff
\begin{equation*}
\sum_{k=1}^n\left[\left(m_k-\displaystyle\frac{\delta}{n}\right)\displaystyle\frac{1}{n}-\displaystyle\frac{\delta}{n^2}e^{-\textrm{i}\theta}-\displaystyle\frac{\textrm {i}h}{n}\right]=-\displaystyle\frac{\delta}{n}e^{-\textrm{i}\theta}-\textrm{i}h=0.
\end{equation*}
This yields
\begin{equation*}
\theta=\theta_{0}=\pi/2,\;\;h=h_{0}=\displaystyle\frac{\delta}{n}.
\end{equation*}
Substituting $h=h_0$ and $\theta=\theta_0$ into Eq. \eqref{subequi}, we see that $\bm z=\bm z_0$.
\end{proof}

Then, we show that
$\bm G(\bm z,\beta,h,\theta,\la)=0$  has a unique solution $(\bm z,\beta,h,\theta)$ of $\bm z\in (X_1)_{\mathbb{C}}$, $\beta\ge0$, $h>0$ and $\theta\in[0,2\pi)$ for small $\la$.
\begin{theorem}\label{cha}
There exist $\la_2>0$ and a continuously differentiable mapping
$\la\mapsto(\bm z_\la,\beta_\la,h_\la,\theta_\la)$ from
$[0,\la_2]$ to $(X_1)_{\mathbb C}\times \mathbb{R}^3$ such that
\begin{equation}\label{3.6G} \begin{cases}
\bm G(\bm z,\beta,h,\theta,\la)=\bm 0,\\
\bm z\in (X_1)_{\mathbb{C}},\;h>0,\;\beta\ge0, \;\theta\in[0,2\pi)\\
\end{cases}
\end{equation}
has a unique solution $(\bm z_\la,\beta_\la,h_\la,\theta_\la)$
for $\la\in[0,\la_2]$.
\end{theorem}
\begin{proof}
Denote the Fr\'echet derivative of $\bm G$ with respect to
$(\bm z,\beta,h,\theta)$ at $(\bm z_{0},\beta_{0},h_{0},\theta_{0},0)$ by $T=(T_1,T_2)^T:(X_1)_{\mathbb C}\times \mathbb{R}^3\mapsto
\mathbb C^n\times \mathbb{R}$.
Then we calculate
\begin{equation*}
\begin{split}
&T_1(\bm \chi,\kappa,\epsilon,\vartheta)=
\left(\begin{array}{c}
\sum_{k=1}^nd_{1k}\chi_k+\left(m_1-\displaystyle\frac{\delta}{n}\right)\displaystyle\frac{\kappa}{n}-\displaystyle\frac{\textrm{i}\epsilon}{n}+\displaystyle\frac{\delta}{n^2}\vartheta
\\
\sum_{k=1}^nd_{2k}\chi_k+\left(m_2-\displaystyle\frac{\delta}{n}\right)\displaystyle\frac{\kappa}{n}-\displaystyle\frac{\textrm{i}\epsilon}{n}+\displaystyle\frac{\delta}{n^2}\vartheta \\
\vdots\\
\sum_{k=1}^nd_{nk}\chi_k+\left(m_n-\displaystyle\frac{\delta}{n}\right)\displaystyle\frac{\kappa}{n}-\displaystyle\frac{\textrm{i}\epsilon}{n}+\displaystyle\frac{\delta}{n^2}\vartheta \\
\end{array}\right),\\
 &T_2(\kappa)=\displaystyle\frac{2\kappa}{n},
\end{split}
\end{equation*}
where $\bm \chi=(\chi_1,\dots,\chi_n)^T\in(X_1)_{\mathbb C}$,
and $T$ is
a bijection from $(X_1)_{\mathbb C}\times \mathbb{R}^3$ to $\mathbb
C^n\times \mathbb{R}$. It follows from the implicit function theorem that
there exist $\la_2>0$ and a continuously differentiable mapping
$\la \mapsto(\bm z_\la,\beta_\la,h_\la,\theta_\la)$ from
$[0,\la_2]$ to $(X_1)_{\mathbb C}\times \mathbb{R}^3$ such that $(\bm z_\la,\beta_\la,h_\la,\theta_\la)$ satisfies
\eqref{3.6G}.

Then we show the uniqueness of the solution of \eqref{3.6G}. Actually, we only need to prove that if $(\bm z^\la,\beta^\la,h^\la,\theta^\la)$ satisfies \eqref{3.6G},
then
$(\bm z^\la,\beta^\la,h^\la,\theta^\la)\rightarrow(\bm z_0,\beta_0,h_0,\theta_0)=\left(\bm z_0,1,\displaystyle\frac{\delta}{n},\displaystyle\frac{\pi}{2}\right)$
as $\la\rightarrow 0$. It follows from Lemma \ref{nu} and \eqref{g1} that $\{h^\la\}, \{\beta^\la\}$ and $\{\theta^\la\}$ are
bounded for $\la\in(0,\la_2]$. Then from the first equation of \eqref{g1}, we see that
\begin{equation*}
\begin{split}
\langle \bm z^\la,D\bm z^\la\rangle =&-\sum_{j=1}^n\left(m_j-u_{\la,j}-e^{-{\textrm i}\theta^\la}u_{\la,j}-ih^\la\right)\displaystyle\frac{\beta^\la
\eta_j\overline {z^\la_j}}{n}\\&-\la\sum_{j=1}^n\left(m_j-u_{\la,j}-e^{-{\textrm i}\theta^\la}u_{\la,j}-ih^\la\right)|z^\la_j|^2.
\end{split}
\end{equation*}
This, combined with Lemma \ref{lem21}, implies that
there exist positive constants $M_1$ and $M_2$ such that
\begin{equation*}
\gamma_2\|\bm z^\la\|^2_{2}\leq |
\langle \bm z^\la,D\bm z^\la \rangle|\le
M_1\|\bm z^\la\|_2+
M_2\la\|\bm z^\la \|^2_{2}
\;\;\text{for}\;\; \la\in(0,\la_2],
\end{equation*} where $\gamma_2$ is defined as in Lemma \ref{lem21}.
Therefore, for sufficiently small $\la_2$, $\{\bm z^\la\}$ is bounded in
$\mathbb C^n$ for $\la\in[0,\la_2]$. Then, for any
 subsequence $\{\la_n\}_{n=1}^\infty$ satisfying $\ds\lim_{n\to\infty}\la_n=0$, there exists a subsequence
 $\{\la_{n_k}\}_{k=1}^\infty$ such that
$$
(\bm z^{\la_{n_k}},\beta^{\la_{n_k}},h^{\la_{n_k}},\theta^{\la_{n_k}})
 \to(\bm z_*, \beta_*,h_*,\theta_*)\;\;\text{in}\;\;\mathbb {C}^n\times \mathbb{R}^3\;\;\text{as}\;\;k\to\infty.$$
Taking the limit of the equation
$$\bm g_1(\bm z^{\la_{n_k}},\beta^{\la_{n_k}},h^{\la_{n_k}},\theta^{\la_{n_k}},\la_{n_k})=0$$
as $k\rightarrow\infty$, we see that
$(\bm z_*, \beta_*,h_*,\theta_*)$ satisfies \eqref{3.5G}, which yields
$$(\bm z_*, \beta_*,h_*,\theta_*)=(\bm z_{0},\beta_{0},h_{0},\theta_{0}).$$ This completes the proof.
\end{proof}

Then from Theorem \ref{cha}, we obtain that:
\begin{theorem}\label{c25}
Let $\Delta(\la,\mu,\tau)$ be defined as in \eqref{triangle}.
Then for $\la\in(0,\la_2],$ $(\nu,\tau,\bm \psi)$ solves
\begin{equation*}
\begin{cases}
\Delta(\la,{\rm i}\nu,\tau)\bm \psi=\bm 0,\\
\nu>0,\;\tau\ge0,\;\bm \psi (\ne\bm 0) \in \mathbb C^n,\\
\end{cases}
\end{equation*}
if and
only if
\begin{equation}\label{par}
\nu=\nu_\la=\la h_\la,\;\bm \psi= a{\bm \psi}_\la,\;
\tau=\tau_{l}=\frac{\theta_\la+2l\pi}{\nu_\la},\;\; l=0,1,2,\cdots,
\end{equation}
where $\bm\psi_\la=\beta_\la\bm \eta+\la \bm z_\la$,
$a$ is a nonzero constant, and
$(\bm z_\la,\beta_\la,h_\la,\theta_\la)$ is defined as in Theorem
\ref{cha}.
\end{theorem}

Now, we show that $\textrm i\nu_\la$ is simple.
\begin{theorem}\label{thm34a}
Let $A_{\tau_{l}}(\la)$ be defined as in \eqref{Ataula}, and assume that
$\la\in(0,\la_2]$, where $\la_2$ is sufficiently small. Then $\rm{i}\nu_\la $ is a simple
eigenvalue of $A_{\tau_{l}}(\la)$ for $l=0,1,2,\dots$, where $\rm i\nu_\la$ and $\tau_l$ are defined as in Theorem \ref{c25}.
\end{theorem}
\begin{proof}
It follows from Theorem \ref{c25} that $\mathscr{N}[A_{\tau_{l}}
(\la)-\rm i\nu_\la]=\text{Span}[e^{\rm{i}\nu_\la\theta}{\bm \psi}_\la]$, where $\theta\in[-\tau_l,0]$ and ${\bm \psi}_\la$ is defined as in
Theorem \ref{c25}.
Then we show that
\begin{equation*}
\mathscr{N}[A_{\tau_{l}}
(\lambda)-\rm i\nu_\la]^2=\mathscr{N}[A_{\tau_{l}}
(\lambda)-\rm i\nu_\la].
\end{equation*} If $\bm \phi\in\mathscr{N}[A_{\tau_{l}}
(\la)-\rm i\nu_\la]^2$,
then
$$
[A_{\tau_{l}}
(\la)-{\rm i}\nu_\la]\bm\phi\in\mathscr{N}[A_{\tau_{l}}(\la)-\textrm i\nu_\la]=
\text{Span}[e^{\textrm i\nu_\la\theta}\bm\psi_\la],
$$
which implies that there exists a constant $a$ such that
$$
[A_{\tau_{l}}
(r)-{\rm i}\nu_\la]\bm\phi=ae^{{\rm i}\nu_\la\theta}\bm\psi_\la.
$$
Then we see that
\begin{equation}
 \label{eq32}
\begin{split}
\dot{\bm \phi}(\theta)&={\rm i}\nu_\la\bm \phi(\theta)+ae^{\rm i\nu_\la\theta}{\bm \psi}_\la,
\ \ \ \ \theta\in[-\tau_{l},0], \\
\dot{\bm \phi}(0)=&D\bm{\phi}(0)+\la \textrm{diag}\left(m_j-u_{\la,j}\right)\bm{ \phi}(0)-\la\textrm{ diag}( u_{\la,j})\bm{ \phi}(-\tau_l).
 \end{split}
 \end{equation}
The first equation of \eqref{eq32} yields
\begin{equation}
\label{eq33}
\begin{split}
\bm \phi(\theta)&=\bm \phi(0)e^{\rm i\nu_\la\theta}+a\theta
e^{\rm i\nu_\la\theta}{\bm \psi}_\la,\\
\dot{{\bm \phi}}(0)&={\rm i}\nu_\la {\bm \phi}(0)+a{\bm \psi}_\la.
\end{split}
\end{equation}
Then, plugging \eqref{eq33} into the second equation of \eqref{eq32}, we have
\begin{equation}\label{provesimple}
\begin{split}
&\Delta(\la,{\rm i}\nu_\la,\tau_{l})\bm \phi(0)\\=&D\bm{\phi}(0)+\la  \textrm{diag}\left(m_j-u_{\la,j}\right)\bm{ \phi}(0)-\la e^{-{\rm i}\theta_\la}\textrm{ diag}( u_{\la,j})\bm{ \phi}(0)-{\rm i}\nu_\la {\bm \phi}(0)
\\=&a\left(\bm \psi_{\la}-\la \tau_l e^{-{\rm i}\theta_\la}\textrm{ diag}( u_{\la,j}){\bm \psi}_\la\right).
\end{split}\end{equation}
Multiplying \eqref{provesimple} by $\left({\psi}_{\la,1},\dots, {\psi}_{\la,n}\right)$ gives
\begin{equation*}
\begin{split}
0=&\left( {\psi}_{\la,1},\dots,{\psi}_{\la,n}\right)\Delta(\la,{\rm i}\nu_\la,\tau_{l})\bm \phi(0)=\left\langle\overline {\bm\psi}_\la,\Delta(\la,{\rm i}\nu_\la,\tau_{l})\bm\phi(0)\right\rangle
\\
=&\left\langle\Delta(\la,-{\rm i}\nu_\la,\tau_{l})\overline {\bm \psi}_\la,\bm \phi(0)\right\rangle=a\left(\sum_{j=1}^n \psi_{\la,j}^2-\la\tau_{l}e^{{\rm-i}\theta_\la}\sum_{j=1}^n u_{\la,j}\psi_{\la,j}^2\right).\\
\end{split}\end{equation*}
Define
\begin{equation}\label{Sn}
S_l(\la):=\sum_{j=1}^n \psi_{\la,j}^2-\la\tau_{l}e^{{\rm-i}\theta_\la}\sum_{j=1}^n u_{\la,j}\psi_{\la,j}^2\;\;\text{for}\;\;l=0,1,2,\dots,
\end{equation}
and we see from Lemma \ref{l25}, Theorems \ref{cha} and \ref{c25} that $\theta_{\la}\to \pi/2$, $\la\tau_l\to \left(\displaystyle\frac{\pi}{2}+2l\pi\right)\displaystyle\frac{n}{\delta}$, and $\psi_{\la,j}\to 1/n$, $u_{\la,j}\to \delta/n$ for any $j=1,\dots,n$ as $\la\to0$, where $\delta$ is defined as in \eqref{del}. Then we obtain that
\begin{equation*}
\lim_{\la\to0}S_l(\la)=\displaystyle\frac{1}{n}\left[1+\textrm i\left(\frac{\pi}{2}+2l\pi\right)\right]\ne0,
\end{equation*}
which implies that
$a=0$. Therefore, for any $l=0,1,2,\dots$,
$$
\mathscr{N}[A_{\tau_{l}}(\la)-{\rm i}\nu_\la]^j
=\mathscr{N}[A_{\tau_{l}}(\la)-{\rm i}\nu_\la],\;\;j=
2,3,\cdots,
$$
and consequently, ${\rm i}\nu_\la$ is a simple eigenvalue of
$A_{\tau_{l}}(\la)$ for $l=0,1,2,\dots$.
\end{proof}

It follows from Theorem \ref{thm34a} that $\mu=\textrm{i}\nu_{\la}$ is a simple eigenvalue of $A_{\tau_{l}}$. Then we see from
the implicit function theorem that there
are a neighborhood $O_{n}\times D_{n}\times
H_{n}$ of
$(\tau_{l},\textrm{i}\nu_\la,{\bm \psi}_\la)$ and a continuously
differential function $(\mu(\tau),\bm\psi(\tau)):O_{n}\rightarrow D_{n}\times
H_{n}$ such that $
\mu(\tau_{l})=\textrm{i}\nu_\la$, $\bm \psi(\tau_{l})={\bm \psi}_\la$, and for each $\tau\in O_{n}$, the only eigenvalue of
$A_\tau(\la)$ in $D_{n}$ is $\mu(\tau),$ and
\begin{equation}\label{eq34}
\begin{split}
\Delta(\la,\mu(\tau),\tau)\bm\psi(\tau)=&D\bm\psi(\tau)-\la \textrm{diag}\left(m_j-u_{\la,j}\right)\bm \psi(\tau)\\
&-\la e^{-\mu(\tau)\tau}\textrm{diag}( u_{\la,j})\bm{ \psi}(\tau)-\mu(\tau)\bm{\psi}(\tau)=0.
\end{split}
\end{equation}
A direct calculation can lead to the transversality
condition.
\begin{theorem}\label{thm35}
For
$\la\in(0,\la_2]$, where $\la_2$ is sufficiently small,
$$
\frac{d{\mathcal Re}[\mu(\tau_{l})]} {d\tau}>0, \;\;l=0,1,2,\cdots.$$
\end{theorem}

\begin{proof}
Differentiate \eqref{eq34} with respect to $\tau$ at $\tau=\tau_l$, we have
\begin{equation}\label{326}
\displaystyle\frac{d\mu(\tau_l)}{d\tau}\left(-\bm \psi_\la+\la\tau_le^{-{\rm i}\theta_\la}\textrm{diag}(u_{\la,j})\bm \psi_\la\right)+\Delta(\la,{\rm i}\nu_\la,\tau_l)\displaystyle\frac{d\bm \psi}{d\tau}+{\rm i}\la\nu_\la e^{-{\rm i}\theta_\la}\textrm{diag}(u_{\la,j})\bm \psi_\la=\bm 0.
\end{equation}
Multiplying \eqref{326} by $\left({\psi}_{\la,1},\dots,
{\psi}_{\la,n}\right)$,  we have
\begin{equation*}
S_n(\la)\displaystyle\frac{d\mu(\tau_l)}{d\tau}+\left\langle \Delta(\la,{-\rm i}\nu_\la,\tau_l) \overline {\bm \psi}_\la,\displaystyle\frac{d\bm \psi}{d\tau}\right\rangle ={\rm i}\la\nu_\la e^{-{\rm i}\theta_\la}\sum_{j=1}^n u_{\la,j}\psi_{\la,j}^2,
\end{equation*}
where $S_n(\la)$ is defined as in \eqref{Sn}. Then
\begin{equation*}
\frac{d\mu(\tau_{l})} {d\tau}=\displaystyle\frac{{\rm i}e^{-{\rm i}\theta_\la}\la\nu_\la\left(\sum_{j=1}^n\psi_{\la,j}^2\right)\sum_{j=1}^nu_{\la,j}\psi_{\la,j}^2}{|S_n(\la)|^2}
-\displaystyle\frac{{\rm i}\la^2\tau_l\nu_\la\left(\sum_{j=1}^nu_{\la,j}\psi_{\la,j}^2\right)^2}{|S_n(\la)|^2}.
\end{equation*}
It follows from Lemma \ref{l25}, Theorems \ref{cha} and \ref{c25} that $\theta_{\la}\to \pi/2$, $\la\tau_l\to \left(\displaystyle\frac{\pi}{2}+2l\pi\right)\displaystyle\frac{n}{\delta}$, $\displaystyle\frac{\nu_\la}{\la}\to\displaystyle\frac{\delta}{n}$, and $\psi_{\la,j}\to 1/n$, $u_{\la,j}\to \delta/n$ for any $j=1,\dots,n$ as $\la\to0$, where $\delta$ is defined as in \eqref{del}. Therefore,
\begin{equation*}
\lim_{\la\to0}
\displaystyle\frac{1}{\la^2}\frac{d{\mathcal Re}[\mu(\tau_{l})]} {d\tau}>0, \;\;l=0,1,2,\cdots,
\end{equation*}
and this completes the proof.
\end{proof}

Finally, from Theorems \ref{c25}, \ref{thm34a} and \ref{thm35}, we obtain the stability/instability of the positive steady state $\bm u_{\la}$, and the existence of the associated Hopf bifurcation.
\begin{theorem}\label{thm37}
Suppose that $(A_1)$ and $(A_2)$ hold, and  $\delta>0$.
Then for $\la\in(0,\la_2]$, where $\la_2$ is sufficiently small,
the unique positive steady state $\bm u_\la$ of \eqref{ndif}
is locally asymptotically stable when $\tau\in[0,\tau_{0})$, and
unstable when $\tau\in(\tau_{0},\infty)$. Moreover, when $\tau=\tau_0$, system \eqref{ndif} occurs Hopf bifurcation at $\bm u_{\la}$.
\end{theorem}

Note that model \eqref{ndif} is equivalent to \eqref{patc}, and $\tau=dr$ and $\la=1/d$, where parameters $d, r$ are in model \eqref{patc}, and parameters $\la,\tau$ are in model \eqref{ndif}. Then we have the following results.
\begin{theorem}\label{equivthm37}
Suppose that $(A_1)$ and $(A_2)$ hold, and $\delta>0$. Then the following results statements hold.
\begin{enumerate}
\item [$(i)$] For any $d>0$, model \eqref{patc} admits a unique positive equilibrium $\bm u^d$.
\item [$(ii)$] For any  $d\in(d_1,\infty]$, where $d_1$ is sufficiently large, there exists $r_0>0$ such
that
the unique positive steady state $\bm u^d$ of \eqref{patc}
is locally asymptotically stable when $r\in[0,r_0)$, and
unstable when $r\in(r_{0},\infty)$. Moreover, when $r=r_0$, system \eqref{patc} occurs Hopf bifurcation at $\bm u^{d}$.
\end{enumerate}
\end{theorem}
\begin{remark}\label{r310}
Since $\tau=dr$ and $\la=1/d$, where parameters $d, r$ are in \eqref{patc}, and parameters $\la,\tau$ are in \eqref{ndif}, we see from Lemma \ref{l25}
and Theorems \ref{cha} and  \ref{c25} that:
\begin{equation*}
\lim_{d\to\infty} r_0=\lim_{d\to\infty}\displaystyle\frac{\tau_0}{d}=\lim_{\la\to0}\la\tau_0=\check r:=\displaystyle\frac{n\pi}{2\sum_{i=1}^n m_i}.
\end{equation*}
Here $\check r$ is also the first Hopf  bifurcation value for the following model:
\begin{equation}\label{average}
w'=w\left(\displaystyle\frac{\sum_{j=1}^n m_j}{n}-w(t-r)\right).
\end{equation}
Therefore, in some sense, the large dispersal rate $d$ may dilute the effect of the connectivity between patches, and when the dispersal rate tends to infinity,
the first Hopf bifurcation value of model \eqref{patc} tends to that of the \lq\lq average\rq\rq~model \eqref{average}.
\begin{remark}
By the similar arguments as that in \cite{ChenWeiZhang}, we can compute that the direction of the Hopf bifurcation at $r=r_0$ is forward and the bifurcating periodic solution from $r=r_0$ is orbitally asymptotically stable. Here we omit the details.
\end{remark}

\end{remark}
\subsection{Case II: $\delta<0$}
In this subsection, we consider the case of $\delta<0$. In this case, $\bm u_\la$ exists and has the form of \eqref{formu} for $\la>\la_*$, where $\la_*$ is defined as in Lemma \ref{lm2.1}.
The arguments are similar as that in the case of $\delta>0$, and therefore we only provide the main results without proof.

We also give an estimates for solutions of Eq. \eqref{eigen} for this case.
\begin{lemma}\label{2nu}
Assume that $(\hat \nu_\la,\hat \theta_\la,\hat {\bm{\psi}}_\la )$ solves \eqref{eigen}, where $\hat{\bm{\psi}}_\la=(\hat\psi_{\la,1},\dots, \hat \psi_{\la,n})^T(\ne\bm 0)
\in\mathbb{C}^n$.  Then for any $\hat  \la_1>\la_*$,
$\left|\ds
\frac{\hat \nu_\la}{\la-\la_*}\right|$ is bounded for
$\la\in(\la_*,\hat\la_1]$, where $\la_*$ is defined as in Lemma \ref{lm2.1}.\end{lemma}

As in Lemma \ref{lem21}, we have
\begin{lemma}\label{2lem21}
Assume that $\hat{\bm z}\in (\hat X_1)_{\mathbb C}$, where $\hat  X_1$ is defined as in \eqref{X1}.
Let
\begin{equation}\label{Dhat}
\hat D=D+\la_*{\rm diag}(m_j),
\end{equation}
 where $\la_*$ is defined as in Lemma \ref{lm2.1}.
Then $|
\langle\hat{\bm z}, \hat D\hat{\bm z}\rangle|\ge
\hat \gamma_2\|\hat{\bm z}\|_2^2$, where $\|\cdot \|_2$ is defined as in \eqref{norm}, and $-\hat\gamma_2(<0)$ is the second
largest eigenvalue of matrix $\hat D$.
\end{lemma}

Assume that $(\hat \nu_\la,\hat \theta_\la,\hat {\bm{\psi}}_\la )$ solves \eqref{eigen} for $\la\in(\la_*,\hat\la_1]$, where $\hat\la_1>\la_*$.
Then ignoring a scalar factor, $\hat{\bm \psi}$ in Eq. \eqref{eigen} can be represented as
\begin{equation}
\label{2eigen2}
\begin{split}
&\hat{\bm \psi}= \hat\beta\hat {\bm \eta}+(\la-\la_*)\hat {\bm  z},\;\;\; \hat{\bm z}\in (\hat X_1)_{\mathbb{C}},\; \; \hat \beta\geq0, \\
 &\|\hat{\bm \psi}\|^2_{2}={\hat \beta}^2 \|\hat{\bm \eta}\|^2_2
 +(\la-\la_*)^2\|\hat{\bm z}\|^2_{2}=  \|\hat{\bm \eta}\|^2_2,
 \end{split}
 \end{equation}
 where $\hat{\bm \eta}$ and $\hat X_1$ are defined as in Lemma \ref{lm2.3}.
Note that $$ {\bm u}_\la=\alpha_\la(\la-\la_*)\left[\hat{\bm\eta}+(\la-\la_*){\bm \xi}_\la\right],$$
where $\alpha_\la$ and $\bm \xi_\la$ are defined as in Lemma \ref{lm2.3}.
Then we see that $(\hat\nu,\hat\theta,\hat{\bm \psi})$ solves Eq. \eqref{eigen}, where $\hat\nu>0$, $\hat\theta\in[0,2\pi)$ and $\hat{\bm \psi}\in \mathbb C^n(\|\hat{\bm \psi}\|^2_{2}= \|\bm \hat \eta\|_2^2)$,
if and only if the following system:
\begin{equation}\label{2g1}
\begin{cases}
\hat{\bm g_1}(\hat{\bm z},\hat\beta,\hat h,\hat\theta,\la)
:=\hat D\hat{\bm{z}}+ \textrm{diag}(m_j)[\hat\beta\hat{\bm \eta}+(\la-\la_*)\hat{\bm z}]-\textrm{i}\hat h[\hat\beta\hat{\bm \eta}+(\la-\la_*)\hat{\bm z}]\\
~~~~~~~~~~~~~~~~~~~~-\la\alpha_\la\textrm{diag}\left(\hat{\eta}_j+(\la-\la_*)\xi_j^\la\right)[\hat\beta\hat{\bm \eta}+(\la-\la_*)\hat{\bm z}]\\
~~~~~~~~~~~~~~~~~~~~-\la\alpha_\la e^{-\textrm i\hat\theta}\textrm{diag}\left(\hat{\eta}_j+(\la-\la_*)\xi_j^\la\right)[\hat\beta\hat{\bm \eta}+(\la-\la_*)\hat{\bm z}]=\bm 0,\\
\hat g_2(\hat{\bm z},\hat \beta,\la):=(\hat\beta^2-1)\|\hat{\bm \eta}\|^2_2+(\la-\la_*)^2\|\hat{\bm z}\|^2_{2}=0\\
\end{cases}
\end{equation}
has a solution $(\hat{\bm z},\hat \beta,\hat h,\hat \theta)$, where $\hat D$ is defined as in \eqref{Dhat}, $\hat{\bm z}\in (\hat X_1)_{\mathbb{C}}$, $\hat \beta\ge0$, $\hat h>0$ and $\hat \theta\in[0,2\pi)$.
Then we define
$\hat{\bm G}:(X_1)_{\mathbb C}\times \mathbb{R}^4\to
\mathbb C^n\times \mathbb{R}$ by $\hat {\bm G}=(\hat{\bm g_1},\hat g_2)^T$. For this case, we will firstly show that
$\hat{\bm G}(\hat{\bm z},\hat \beta,\hat h,\hat \theta,\la )=0$ has a unique solution for $\la=\la_*$.
\begin{lemma}\label{2l25}
The following equation \begin{equation}\label{23.5G}
\begin{cases}
\hat{\bm G}(\hat{\bm z},\hat\beta,\hat h,\hat \theta,\la_*)={\bm 0},\\
\hat{\bm z}\in (X_1)_{\mathbb{C}},\;\hat h\ge0,\;\hat \beta\ge0,\; \hat \theta\in[0,2\pi]\\
\end{cases}
\end{equation} has a unique solution $(\hat {\bm z}_{\la_*},\hat \beta_{\la_*},\hat h_{\la_*},\hat \theta_{\la_*})$. Here
\begin{equation}\label{2lastar}
    \hat\beta_{\la_*}=1,\;\;\hat\theta_{\la_*}=\pi/2,\;\;\hat h_{\la_*}=\displaystyle\frac{\sum_{j=1}^nm_j\hat \eta_j^2}{\sum_{j=1}^n\hat \eta_j^2},
\end{equation}
and $\hat {\bm z}_{\la_*}=(\hat z_{\la_*,1},\dots,\hat z_{\la_*,n})^T\in(\hat X_1)_{\mathbb C}$ is the unique solution of the following equation
\begin{equation}\label{2z0}
\sum_{k=1}^n d_{jk} \hat z_k+\la_*m_j\hat z_j+m_j\hat \eta_j-{\rm i}\hat h_{\la_*}\hat \eta_j-\la_*\alpha_{\la_*}\hat\eta_j^2+{\rm i}\la_*\alpha_{\la_*}\hat\eta_j^2=0, \;\;j=1,\dots,n.
\end{equation}
where $\alpha_{\la_*}$ and $\hat{\bm \eta}$ are defined in Lemma \ref{lm2.3}.
\end{lemma}

\begin{proof}
A direct computation implies that $\hat g_2(\hat{\bm z},\hat\beta,\la_*)=0$ iff $\hat \beta=\hat \beta_{\la_*}=1$. Note that
\begin{equation}\label{22}
 \alpha_{\la_*}=\displaystyle\frac{\sum_{j=1}^nm_j\hat\eta_j^2}{\la_*\sum_{j=1}^n\hat\eta_j^3}
\end{equation}
from \eqref{lims2}.
Then substituting $\hat\beta=\hat\beta_{\la_*}=1$ into
$\hat{\bm g}_1(\hat{\bm z},\hat\beta,\hat h,\hat\theta,\la_*)=\bm 0$, we see that
\begin{equation}\label{2subequi}
-\sum_{k=1}^n d_{jk}\hat z_k-\la_*m_j\hat z_j=m_j\hat\eta_j-{\rm i}\hat h\hat\eta_j-\la_*\alpha_{\la_*}\hat\eta_j^2-\la_*\alpha_{\la_*}e^{-{\rm i}\hat\theta}\hat\eta_j^2, \;\;j=1,\dots,n.
\end{equation}
Then Eq. \eqref{2subequi} has a solution $(\hat{\bm z},\hat h,\hat \theta)$, where $\hat{\bm z}\in(\hat X_1)_{\mathbb C}$, $\hat h\ge0$, $\hat\theta\in[0,2\pi]$, iff
\begin{equation*}
\sum_{k=1}^n\left(m_j\hat\eta_j^2-\la_*\alpha_{\la_*}\hat\eta_j^3-{\rm i}\hat h\hat\eta_j^2-\la_*\alpha_{\la_*}e^{-{\rm i}\hat\theta}\hat\eta_j^3\right)=0.
\end{equation*}
This, combined with \eqref{22}, implies that
\begin{equation*}
\hat\theta=\hat\theta_{\la_*}=\pi/2,\;\;\hat h=\hat h_{\la_*}.
\end{equation*}
Substituting $\hat h=\hat h_{\la_*}$ and $\hat \theta=\hat \theta_{\la_*}$ into Eq. \eqref{2subequi}, we see that $\hat{\bm z}=\hat {\bm z}_{\la_*}$.
\end{proof}

Then, we can also
show that
$\hat{\bm G}(\hat{\bm z},\hat\beta,\hat h,\hat \theta,\la)=0$  has a unique solution by virtue of the
 implicit function theorem.
\begin{theorem}\label{2cha}
There exist $\hat\la_2>\la_*$ and a continuously differentiable mapping
$\la\mapsto(\hat{\bm z}_\la,\hat\beta_\la,\hat h_\la,\hat\theta_\la)$ from
$[\la_*,\la_2]$ to $(\hat X_1)_{\mathbb C}\times \mathbb{R}^3$ such that
\begin{equation}\label{23.6G} \begin{cases}
\hat{\bm G}(\hat{\bm z},\hat\beta,\hat h,\hat \theta,\la)=\bm 0,\\
\hat{\bm z}\in (\hat X_1)_{\mathbb{C}},\;\hat h>0,\;\hat \beta\ge0, \;\hat \theta\in[0,2\pi)\\
\end{cases}
\end{equation}
has a unique solution $(\hat{\bm z}_\la,\hat \beta_\la,\hat h_\la,\hat \theta_\la)$
for $\la\in[\la_*,\la_2]$.
\end{theorem}

It follows from Theorem \ref{2cha} that:
\begin{theorem}\label{2c25}
Let $\Delta(\la,\mu,\tau)$ be defined as in \eqref{triangle}. Then for $\la\in(\la_*,\hat \la_2],$ $(\hat\nu,\hat \tau,\hat{\bm \psi})$ solves
\begin{equation*}
\begin{cases}
\Delta(\la,{\rm i}\hat \nu,\hat \tau)\hat{\bm \psi}=\bm 0,\\
\hat \nu>0,\;\hat \tau\ge0,\;\hat{\bm \psi} (\ne \bm 0) \in \mathbb C^n,\\
\end{cases}
\end{equation*}
if and
only if
\begin{equation}\label{2par}
\hat \nu=\hat \nu_\la=(\la -\la_*)\hat h_\la,\;\hat{\bm \psi}= a{\bm \psi}_\la,\;
\hat \tau=\hat \tau_{l}=\frac{\hat\theta_\la+2l\pi}{\hat\nu_\la},\;\; l=0,1,2,\cdots,
\end{equation}
where $\hat{\bm \psi}_\la=\hat\beta_\la\hat{\bm \eta}+(\la -\la_*)\hat{\bm z}_\la$,
$a$ is a nonzero constant, and
$(\hat{\bm z}_\la,\hat\beta_\la,\hat h_\la,\hat \theta_\la)$ is defined as in Theorem
\ref{2cha}.
\end{theorem}

We can also show that ${\rm i}\hat \nu_\la$ is simple and the transversality condition holds as in the case of $\delta>0$.
Then the stability and the existence of Hopf bifurcation can be summarized as follows.
\begin{theorem}\label{2thm37}
Suppose that $(A_1)$ and $(A_2)$ hold, and $\delta<0$. Then
for any $\la\in(\la_*,\hat\la_2]$, where $0<\hat\la_2-\la_*\ll 1$,
the unique positive steady state $\bm u_\la$ of \eqref{ndif}
is locally asymptotically stable when $\tau\in[0,\hat \tau_{0})$, and
unstable when $\tau\in(\hat \tau_{0},\infty)$. Moreover, when $\tau=\hat\tau_0$, system \eqref{ndif} occurs Hopf bifurcation at $\bm u_{\la}$.
\end{theorem}

Then, for the equivalent model \eqref{patc}, we have the following results.
\begin{theorem}\label{2equivthm37}
Suppose that $(A_1)$ and $(A_2)$ hold, and $\delta<0$. Then the following results statements hold.
\begin{enumerate}
\item [$(i)$] For any $d\in(0,d_*)$, where $d_*=1/\la_*$, model \eqref{patc} admits a unique positive equilibrium $\bm u^d$.
\item [$(ii)$] For any  $d\in(\hat d_1,d_*)$, where $0<d_*-\hat d_1\ll 1$, there exists $\hat r_0>0$ such
that
the unique positive steady state $\bm u^d$ of \eqref{patc}
is locally asymptotically stable when $r\in[0,\hat r_0)$, and
unstable when $r\in(\hat r_{0},\infty)$. Moreover, when $r=\hat r_0$, system \eqref{patc} occurs Hopf bifurcation at $\bm u^{d}$.
\end{enumerate}
\end{theorem}
\begin{remark}
Note that $\tau=dr$ and $\la=1/d$, where parameters $d, r$ are in \eqref{patc}, and parameters $\la,\tau$ are in \eqref{ndif}. For this case, we see from Lemma \ref{2l25}
and Theorems \ref{2cha} and \ref{2c25} that:
\begin{equation*}
\lim_{d\to d_*} (d_*-d)\hat r_0=\lim_{\la\to\la_*}\displaystyle\frac{(\la-\la_*)\hat\tau_0}{\la_*}=\displaystyle\frac{d_*\pi}{2\hat h_{\la_*}},
\end{equation*}
which implies that $\lim_{d\to d_*} r_0=\infty$. This case is different from the case of $\delta>0$.

\end{remark}
\begin{remark}
By the similar arguments as that in \cite{ChenLouWei}, we can also compute that the direction of the Hopf bifurcation at $r=\hat r_0$ is forward and the bifurcating periodic solution from $r=\hat r_0$ is orbitally asymptotically stable. Here we omit the details.
\end{remark}

\section{Small dispersal rate}
In Section 2, we show the existence of Hopf bifurcation of model \eqref{patc} when the dispersal rate $d$ is large (respectively, $d$ is near $d_*=1/\la_*$) for $\delta>0$ (respectively, $\delta<0$).
In this section, we consider the case that $d$ is small.
Here we need another assumption:
\begin{enumerate}
\item [($A'_3$)] $m_j\ne 0$ for any $j=1,\dots,n$.
\end{enumerate}
Without loss of generality, $(A_2)$ and $(A'_3)$ can be replaced by the following assumption:
\begin{enumerate}
\item [($A_3$)] $m_j> 0$ for $j=1,\dots,p$, and $m_j<0$ for $j=p+1,\dots,n$, where $1\le p\le n$.
\end{enumerate}
Therefore, we assume that $(A_1)$ and $(A_3)$ hold throughout this section. It follows from Lemma \ref{lm2.3} that model \eqref{patc} admits a unique positive equilibrium
$\bm u^{d}=(u_1^d,\dots,u_n^d)^T\gg\bm 0$ for $d\in(0,\hat d)$, where
\begin{equation}\label{hatd}
\hat d=\begin{cases}
\infty, &\text{if}\;\; \delta\ge0,\\
d_*=1/\lambda_*,&\text{if}\;\; \delta<0.
\end{cases}
\end{equation}
 Moreover, we denote $(x)_+=0$ if $x\le 0$ and $(x)_+=x$ if $x>0$.

We firstly show the asymptotic profile of $\bm u^{d}$ as $d\to0$.
\begin{lemma}\label{asymp}
Assume that $(A_1)$ and $(A_3)$ hold, and let $\bm u^{d}=(u^d_{1},\dots,u^d_{n})\gg\bm0$ be the unique positive equilibrium of model \eqref{patc} for $d\in(0,\hat d)$, where $\hat d$ is defined as in \eqref{hatd}. Then
$\lim_{d\to0} u^d_{j}=(m_j)_+$ for any $j=1,\dots,n$. Moreover, let
$$\bm u^{0}=((m_1)_+,\dots,(m_n)_+),$$ and
$\bm u^{d}$ is continuously differentiable for $d\in[0,\hat d)$.
\end{lemma}
\begin{proof}
Define \begin{equation*}
\begin{split}
&F(d,\bm u)=
\left(\begin{array}{c}
d\sum_{k=1}^nd_{1k}u_k+u_1(m_1-u_1)
\\
d\sum_{k=1}^nd_{2k}u_k+u_2(m_2-u_2) \\
\vdots\\
d\sum_{k=1}^nd_{nk}u_k+u_n(m_n-u_n) \\
\end{array}\right),
\end{split}
\end{equation*}
and $\bm u^*=(u_1^*,\dots,u_n^*)$, where $u_j^*=(m_j)_+$ for any $j=1,\dots,n$.

Clearly, $F(0,\bm u^*)=\bm 0$, and $D_{\bm {u}} \bm F(0,\bm {u}^{*})=\textrm {diag}(\delta_j)$, where $D_{\bm {u}} \bm F(0,\bm {u}^{*})$ is the Fr\'echet derivative of $F(d,\bm u)$ with respect to $u$ at $(0,\bm {u}^{*})$, and
\begin{equation*}
\delta_j=\begin{cases} -m_j,&j=1,\dots,p,\\
m_j,&\;\;j=p+1,\dots,n.
\end{cases}
\end{equation*}
Therefore, $D_{\bm {u}} \bm F(0,\bm {u}^{*})$ is invertible. It follows from the implicit function theorem that
there exist $ d_1>0$ and a continuously differentiable mapping
$$d\in[0, d_1]\mapsto \tilde {{\bm u}}(d)=({\tilde u}_1(d),\dots,{\tilde u}_n(d))^T\in\mathbb R^n$$ such that $\bm F(d,\tilde {{\bm u}}(d))=\bm 0$ and
$\tilde {{\bm u}}(0)=\bm u^*$.

Taking the derivative of $\bm F(d,\tilde{{\bm  u}}(d))=\bm 0$ with respect to $d$ at $d=0$, we have
\begin{equation*}
-\textrm {diag}(\delta_j)\left(\begin{array}{c} {\tilde u}'_1(0)\\
 {\tilde u}'_2(0)\\
 \vdots\\
  {\tilde u}'_n(0)\end{array}\right)=D\left(\begin{array}{c} u^*_1\\
 u^*_2\\
 \vdots\\
  u^*_n\end{array}\right).
\end{equation*}
Then
\begin{equation*}
\left(\begin{array}{c} {\tilde u}'_1(0)\\
 {\tilde u}'_2(0)\\
 \vdots\\
  {\tilde u}'_n(0)\end{array}\right)=-\textrm {diag}(1/\delta_j)D\left(\begin{array}{c} u^*_1\\
 u^*_2\\
 \vdots\\
  u^*_n\end{array}\right).
\end{equation*}
Since  $\tilde {\bm u}(0)>\bm 0$, then ${\tilde u}'_j(0)\ge0$ if $\tilde u_j(0)=0$, which implies
$\tilde{\bm u}(d)$ is the nonnegative equilibrium of model \eqref{patc} for small $d$. It follows from Lemma \ref{lm2.3} that $\tilde {\bm u}(d)=\bm u^d$, $\lim_{d\to0} u_j^d=(m_j)_+$ for any $j=1,\dots,n$, and $\bm u^{d}$ is continuously differentiable for $d\in[0,\infty)$.
\end{proof}
As in Section 2, linearizing \eqref{patc} at $\bm{u}^d$, we have
\begin{equation}
\label{2linear}
\displaystyle\frac{d \bm {v}}{d t} =dD\bm{ v}+{\rm diag}\left(m_j-u^d_{j}\right)\bm v- {\rm diag}( u^d_{j})\bm{ v}(t-r).
\end{equation}
Then the infinitesimal generator $\tilde A_r(d)$ of the solution semigroup of
\eqref{2linear} is defined by
\begin{equation}\label{2Ataula}
\tilde A_r(d) \bm{\Psi}=\dot{\bm{\Psi}},\end{equation}
and the domain of $\tilde A_r(d)$ is
\begin{equation*}
\begin{split}
 &\mathscr{D}(\tilde A_r(d)) = \big\{\bm {\Psi}\in C_\mathbb{C}
\cap C^1_\mathbb{C}:\ \bm \Psi(0)\in \mathbb{C}^n,\dot{\bm \Psi}(0)=dD\bm{\Psi}(0)+\textrm{diag}\left(m_j-u^d_{j}\right)\bm{ \Psi}(0)\\
&~~~~~~~~~~~~~~~~~-\textrm{ diag}( u^d_{j})\bm{ \Psi}(-r) \big\},
\end{split}
\end{equation*}
where $C_{\mathbb C}=C([-r,0],\mathbb{C}^n)$ and $C^1_\mathbb{C}=C^1([-r,0],\mathbb{C}^n)$. Then, $\mu\in\mathbb{C}$ is an eigenvalue of $\tilde A_r(d)$ iff there exists $\bm{\psi}=(\psi_1,\dots,\psi_n)^T(\ne\bm 0)\in\mathbb{C}^n$ such that $\tilde \Delta(d,\mu,r)\bm \psi=\bm 0$,
where
\begin{equation}\label{2triangle}
\begin{split}
&\tilde\Delta(d,\mu,r)\bm{\psi}
:=dD\bm{\psi}+ \textrm{diag}\left(m_j-u^d_{j}\right)\bm \psi- e^{-\mu r}\textrm{diag}( u^d_{j})\bm{ \psi}-\mu\bm{\psi}.
\end{split}
\end{equation}
Therefore, $\tilde A_r(d)$ has a purely imaginary eigenvalue $\mu=\textrm i\nu\
(\nu>0)$ for some $r\ge0$, iff
\begin{equation}\label{2eigen}
\begin{split}
H(d,\nu,\theta, \bm\psi):=dD\bm{\psi}+\textrm{diag}\left(m_j-u^d_{j}\right)\bm \psi-e^{-\textrm i\theta}\textrm{diag}( u^d_{j})\bm{ \psi}-\textrm{i}\nu\bm{\psi}=\bm 0
\end{split}
\end{equation}
is solvable for some value of $\nu>0$, $\theta\in[0,2\pi)$, and $\bm{\psi}(\ne \bm 0)\in \mathbb{C}^n$.

As in Section 2, we also give an estimate for solutions of Eq. \eqref{2eigen}.
\begin{lemma}\label{3nu}
Assume that $(\nu^d, \theta^d,\bm{\psi}^d )$ solves \eqref{2eigen}, where $\bm{\psi}^d=(\psi_1^d,\dots,\psi_n^d)^T(\ne\bm 0)
\in\mathbb{C}^n$.  Then for any $\tilde d>0$,
$\left|\nu^d\right|$ is bounded for
$d\in(0,\tilde d]$.
\end{lemma}

\begin{proof}
Substituting $(\nu^d,\theta^d,\bm{\psi}^d )$ into \eqref{2eigen} and multiplying it by $\left(\overline {\psi_1^d},\dots,\overline
{\psi_n^d}\right)$, we have
\begin{equation*}
\left(\overline {\psi_1^d},\dots,\overline
{\psi_n^d}\right)\left[dD{\bm \psi}^d + \textrm{diag}\left(m_j-u_j^d\right)\bm \psi^d- e^{-\textrm i\theta^d}\textrm{diag}( u_j^d)\bm{ \psi}^d-\textrm{i}\nu^d\bm{\psi}^d\right]=0,
\end{equation*}
Using the similar arguments as in the proof of Lemma \ref{nu},
we obtain that
\begin{equation*}
\nu^d\sum_{j=1}^n\left|\psi^d_j\right|^2= \sin\theta^d\sum_{j=1}^nu_j^d\left|\psi^d_j\right|^2,
\end{equation*}
which implies that
\begin{equation}
\left|\nu^d\right|\le\max_{d\in[0,\tilde d]}\|\bm u^d\|_\infty \;\;\text{for any}\;\;d\in(0,\tilde d],
\end{equation}
where $\|\cdot\|_\infty$ is defined as in \eqref{norm}. Then we see from Lemma \ref{asymp} that
 $\left|\nu^d\right|$ is bounded for
$d\in(0,\tilde d]$.
\end{proof}

Now we consider the solution of \eqref{2eigen} for $d=0$.

\begin{lemma}\label{limz}
Assume that $(A_1)$ and $(A_3)$ hold, $d=0$, and $m_i\ne m_j$ for any $i\ne j$ and $1\le i,j\le p$. Then there exist exactly $p$ pairs of $(\nu_q^0,\theta_q^0)\in \mathbb R_+\times [0,2\pi]$, where $\nu_q^0=m_q$, $\theta_q^0=\pi/2$ for $q=1,\dots,p$, such that
\begin{equation}
\mathcal S_q:=\{\bm \psi:\bm \psi\ne \bm 0, \;H(0, \nu_q^0,\theta_q^0,\bm\psi)=\bm 0\}\ne\emptyset,
\end{equation}
where $H(d,\nu,\theta, \bm\psi)$ is defined as in \eqref{2eigen}.
Moreover, for any $q=1,\dots,p$, $\mathcal S_q=\{c\bm\psi_q^0:c\in\mathbb C\}$, where
${\bm \psi}^{0}_q=(\psi^{0}_{q1},\dots,\psi^{0}_{qn})$,
$\psi^{0}_{qq}=1$ and $\psi^{0}_{qj}=0$ for $j\ne q$.
\end{lemma}
\begin{proof}
Note that
\begin{equation*}
\bm u^0=\begin{cases} m_j,&j=1,\dots,p,\\
0,&\;\;j=p+1,\dots,n.
\end{cases}
\end{equation*}
Therefore, if there exists $\bm \psi\ne\bm 0$ such that
$H(0, \nu,\theta,\bm\psi)=\bm0$, then
\begin{equation*}
\prod_{j=1}^p(m_je^{-\textrm{i}\theta}+\textrm {i}\nu)\prod_{j=p+1}^n(m_j-\textrm{i}\nu)=0,
\end{equation*}
which implies that $\nu=\nu_q^0=m_q$, $\theta=\theta_q^0=\pi/2$ for $q=1,\dots,p$. Since $m_i\ne m_j$ for any $i\ne j$ and $1\le i,j\le p$,
it follows that $\mathcal S_q=\{c\bm\psi_q^0:c\in\mathbb C\}$. This completes the proof.
\end{proof}

Then we show that:
\begin{lemma}\label{l4.4}
Assume that  $(A_1)$ and $(A_3)$ hold, $m_i\ne m_j$ for any $i\ne j$ and $1\le i,j\le p$, and $d\in(0,\tilde d)$, where $\tilde d$ is sufficiently small. Then there exist exactly $p$ pairs of $(v_q^d,\theta_q^d)\in\mathbb R_+\times [0,2\pi]$ such that
\begin{equation}\label{4.8}
\mathcal S^d_q:=\{\bm \psi:\bm \psi\ne \bm 0,\;H(d, \nu_q^d,\theta_q^d,\bm\psi)=\bm 0\}\ne\emptyset,
\end{equation}
where $H(d,\nu,\theta, \bm\psi)$ is defined as in \eqref{2eigen}.
Moreover, for any $q=1,\dots,p$, $\mathcal S^d_q=\{c\bm\psi_q^d:c\in\mathbb C\}$, and
\begin{equation*}\label{limd}
\lim_{d\to0} v_q^d= v_q^0=m_q, \;\;\lim_{d\to0} \theta_q^d= v_q^0=\pi/2 \;\;\text{and}\;\;\lim_{d\to0}\bm \psi_q^d=\bm \psi_q^0,
\end{equation*}
 where $\bm \psi_q^0$ is defined as in Lemma \ref{limz}.
\end{lemma}

\begin{proof}
We firstly show the existence. For simplicity, we will only show the existence of $(v_1^d,\theta_1^d)$, and the others could be obtained similarly.
Let $$Y_1:=\{\bm x=(x_1,\dots,x_n)^T\in\mathbb {C}^n:x_1=0\},$$
 and consequently $\mathbb C^n={\rm span}\{\bm\psi_1^0\}\oplus Y_1$. Let
$$H_1(d, \nu,\theta,\bm \xi_1):=H(d, \nu,\theta,\bm\psi^0_1+\bm \xi_1):\mathbb R^3\times Y_1\to \mathbb C^n.$$
A direct computation implies that
$H_1(0,\nu_1^0,\theta_1^0,\bm 0)=\bm 0$, and
\begin{equation*}
\begin{split}
&D_{(\nu,\theta,\bm \xi)}H_1(0,\nu_1^0,\theta_1^0,\bm 0)[\vartheta,\epsilon,\bm \chi]=
\left(\begin{array}{c}
m_1\epsilon-{\rm i}\vartheta
\\
{\rm i}(m_2-m_1) \chi_2\\
\vdots\\
{\rm i}(m_p-m_1) \chi_p\\
(m_{p+1}-{\rm i}m_1)\chi_{p+1}\\
\vdots\\
(m_{n}-{\rm i}m_1)\chi_{n}\\
\end{array}\right),
\end{split}
\end{equation*}
where $\bm \chi=(\chi_1,\dots,\chi_n)\in Y_1$, and $D_{(\nu,\theta,\bm \xi)}H_1(0,\nu_1^0,\theta_1^0,\bm 0)$ is the Fr\'echet derivative of $H_1$ with respect to
$(\nu,\theta,\bm \xi_1)$ at $(0,\nu_1^0,\theta_1^0,\bm 0)$. Note that $D_{(\nu,\theta,\bm \xi)}H_1(0,\nu_1^0,\theta_1^0,\bm 0)$ is a bijection.
It follows from the implicit function theorem
that there exist a constant $\delta>0$, a neighborhood $N_1$ of $(\nu_1^0,\theta_1^0,\bm 0)$ and a continuously differentiable function
$$
(v_1^d,\theta_1^d,\bm \xi_1^d):[0,\delta)\mapsto N_1
$$
such that for any $d\in [0,\delta)$, the unique solution of $H_1(d, \nu,\theta,\bm \xi)=\bm 0$ in the neighborhood $N_1$ is $(v_1^d,\theta_1^d,\bm \xi_1^d)$. Letting $\bm \psi_1^d=\bm \psi_1^0+\bm \xi_1^d$, we see that
 \begin{equation}\label{SSp}
\textrm{span}(\bm \psi_1^d)\subset \mathcal S^d_1 \;\;\text{for any}\;\; d\in[0,\delta),
\end{equation}
and \eqref{4.8} holds.
Since the dimension of $\mathcal S^d_1$ is upper semicontinuous, it follows that there exists $\delta_1<\delta$
such that $\dim \mathcal S_1^d\le 1$ for any $d\in[0,\delta_1)$. This, together with \eqref{SSp}, implies that
$\mathcal S^d_1=\{c\bm\psi_1^d:c\in\mathbb C\}$. This completes the part of existence.

Now we show that there exists $\delta_2<\delta_1$ such that, for any $d\in(0,\delta_2)$,
there are exactly $p$ pairs of $(v_q^d,\theta_q^d)\in\mathbb R_+\times [0,2\pi]$ such that \eqref{4.8} holds.
If it is not true, then there exist sequences $\{d_j\}_{j=1}^\infty$
and $\left\{\left(\nu^{d_j},\theta^{d_j},\bm \psi^{d_j}\right)\right\}_{j=1}^\infty$
such that
$\lim_{j\to\infty}d_j=0$, and for any $j$, $\left(\nu^{d_j},\theta^{d_j}\right)\ne (v_q^{d_j},\theta_q^{d_j})(q=1,\dots,p)$, $\left\|\bm \psi^{d_j}\right\|_2=1$, $\nu^{d_j}\ge0$, $\theta^{d_j}\in[0,2\pi]$, and
$$H\left(d_j, \nu^{d_j},\theta^{d_j},\bm \psi^{d_j}\right)=\bm 0.$$
It follows from Lemma \ref{3nu} that $\{\nu^{d_j}\}$ is bounded.
Then
there exists a subsequence $\left\{\left(\nu^{d_{j_k}},\theta^{d_{j_k}},\bm \psi^{d_{j_k}}\right)\right\}_{j=1}^\infty$ such that
\begin{equation*}
\lim_{k\to \infty} \theta^{d_{j_k}}=\theta^*,\;\;\lim_{k\to\infty}\nu^{d_{j_k}}=\nu^*\;\;\text{and}\;\;\lim_{k\to\infty}\bm\psi^{d_{j_k}}=\bm\psi^*.
\end{equation*}
This implies that $H(0,\nu^*,\theta^*,\bm\psi^*)=\bm 0$, and consequently we see from Lemma \ref{limz} that there exist $1\le q_0\le p$ and a constant $c_{q_0}$ such that
$\nu^*=\nu_{q_0}^0$ and $\theta^*=\theta_{q_0}^0$, $\bm\psi^*=c_{q_0}\bm \psi_{q_0}^0$.
 Without loss of generality, we assume that $q_0=1$. Then for sufficiently large $k$,
 $$\left(\nu^{d_{j_k}},\theta^{d_{j_k}},\displaystyle\frac{1}{c_1}\bm \psi^{d_{j_k}}-\bm \psi_1^0\right)\in N_1,$$
 which yields
 $$\left(\nu^{d_{j_k}},\theta^{d_{j_k}}\right)=\left(\nu_1^{d_{j_k}},\theta_1^{d_{j_k}}\right).$$
This is a contradiction. Therefore, there exist exactly $p$ pairs of $(v_q^d,\theta_q^d)\in\mathbb R_+\times [0,2\pi]$ such that
\eqref{4.8} holds.
\end{proof}

Then we see from Lemma \ref{l4.4} that:
\begin{theorem}\label{3c25}
Assume that $(A_1)$ and $(A_3)$ hold, $m_i\ne m_j$ for any $i\ne j$ and $1\le i,j\le p$, and $d\in(0,\tilde d)$, where $\tilde d$ is sufficiently small. Then  $(\nu,\tau,\bm \psi)$ solves
\begin{equation*}
\begin{cases}
\tilde \Delta(d,{\rm i}\nu,r)\bm \psi=\bm0,\\
\nu>0,\;r\ge0,\;\bm \psi (\ne \bm0) \in \mathbb C^n,\\
\end{cases}
\end{equation*}
if and
only if there exists $1\le q\le p$ such that
\begin{equation}\label{3par}
\nu=\nu^d_q,\;\bm \psi= c{\bm \psi}^d_q,\;
r=r_{ql}=\frac{\theta^d_q+2l\pi}{\nu^d_q},\;\; l=0,1,2,\cdots,
\end{equation}
where $\nu^d_q$, $\theta^d_q$, and ${\bm \psi}^d_q$ are defined as in Lemma \ref{l4.4}.
\end{theorem}
Finally, we show the existence of Hopf bifurcation.
\begin{theorem}
Assume that assumptions $(A_1)$ and $(A_3)$ hold, $m_i\ne m_j$ for any $i\ne j$ and $1\le i,j\le p$, and let $\bm u^{d}=(u^d_{1},\dots,u^d_{n})^T\gg\bm 0$ be the unique positive equilibrium of model \eqref{patc} for $d\in(0,\hat d)$, where $\hat d$ is defined as in \eqref{hatd}. Then for
any $d\in(0,\tilde d)$, where $\tilde d$ is sufficiently small,
the unique positive steady state $\bm u^d$ of \eqref{patc}
is locally asymptotically stable when $r\in[0,\tilde r_0)$, and
unstable when $r\in(\tilde r_0,\infty)$, where $\tilde r_0=r_{q_*0}$ and $q^*$ satisfies $m_{q^*}=\max_{1\le j\le p} m_j$. Moreover, when $r=\tilde r_0$,
system \eqref{patc} occurs Hopf bifurcation at $\bm u^{d}$.
\end{theorem}
\begin{proof}
It follows from Lemma \ref{l4.4} and Eq. \eqref{3par} that
\begin{equation}
\label{dsmall}
\lim_{d\to0}r_{q0}=\lim_{d\to0}\displaystyle\frac{\theta_q^d}{\nu_q^d}=\displaystyle\frac{\pi}{2m_q}\;\;\text{for any}\;\;q=1,\dots,p.
\end{equation}
Therefore $r_{q_{*}0}=\min\{r_{10},r_{20},\dots,r_{p0}\}$ when $d$ is sufficiently small.
Let
\begin{equation}\label{Snd}
S_{q_*}(d):=\sum_{j=1}^n \left(\psi^d_{q_*j}\right)^2-r_{q_*0}e^{{\rm-i}\theta_{q_*}^d}\sum_{j=1}^n u_{j}^d\left(\psi^d_{q_*j}\right)^2.
\end{equation}
It follows from Lemmas \ref{limz} and \ref{l4.4} that
$$\lim_{d\to0}S_{q_*}(d)=1+\displaystyle\frac{\pi}{2}\rm{ i}.$$
Using the similar arguments as in the proof of Theorems \ref{thm34a} and \ref{thm35}, we see that when $r=r_{q_*0}$, \eqref{2triangle} has
a pair of simple eigenvalue
$\pm {\rm i}m_{q_*}$, and the transversality condition holds. Therefore, when $r=r_{q_*0}$,
system \eqref{patc} occurs Hopf bifurcation at $\bm u^{d}$. This completes the proof.
\end{proof}

\begin{remark}\label{remark3.7}
It follows from \eqref{dsmall} that the first Hopf bifurcation $\tilde r_0$ satisfies
\begin{equation*}
\lim_{d\to0}\tilde r_{0}=\hat r_{q_*}:=\displaystyle\frac{\pi}{2m_{q_*}},
\end{equation*}
where $m_{q^*}=\max_{1\le j\le p} m_j$.
Here $\hat r_{q_*}$ is also the first Hopf bifurcation value for the following model:
\begin{equation}\label{localm}
w'=w\left(m_{q_*}-w(t-r)\right),
\end{equation}
where $m_{q^*}=\max_{1\le j\le p} m_j$.

\end{remark}
\begin{remark}
By the similar arguments as that in \cite{ChenLouWei}, we can also compute that the direction of the Hopf bifurcation at $r=\tilde r_0$ is forward and the bifurcating periodic solution from $r=\tilde r_0$ is orbitally asymptotically stable. Here we also omit the details.
\end{remark}
\section{Discussion}
In this section, we summarize the connection between the Hopf bifurcation of the patch model \eqref{patc} and that
of the \lq\lq local\rq\rq~model:
\begin{equation}\label{localj}
u_j'=u_j\left(m_{j}-u(t-r)\right).
\end{equation}
For simplicity, we only show the case that $m_j>0$ for any $j=1,\dots,p$, and $m_i\ne m_j$ for any $i\ne j$. That is, each patch is favorable for the species.
A direct computation implies that for each $j$, the first Hopf bifurcation value of model \eqref{localj} is
$\hat r_{j}=\displaystyle\frac{\pi}{2m_{j}}$. We call it the \lq\lq local\rq\rq~Hopf bifurcation value.
Then it follows from Remark \ref{remark3.7} that when the dispersal rate $d$ tends to zero, the first Hopf bifurcation of the patch model \eqref{patc} tends to the minimum value of the \lq\lq local\rq\rq~Hopf bifurcation value over all patch $j$. On the other hand, we see from Remark \ref{r310} that the large dispersal rate $d$ may dilute the effect of the connectivity between patches, and
the first Hopf bifurcation value of model \eqref{patc} tends to that of the \lq\lq average\rq\rq~model \eqref{average} when the dispersal rate $d$ tends to infinity. For model \eqref{dclass}, we see from \cite{ShiShi2019} that when $d$ is large or near some critical value, delay $r$ could induce Hopf bifurcation, and similar results could also be found in \cite{ChenLouWei,ChenShi2012,ChenWeiZhang,ChenYu2016,Guo2015,Guo2017,GuoYan2016,HuYuan2011,SuWeiShi2009,SuWeiShi2012,YanLi2010,YanLi2012} for other delayed reaction-diffusion models. It is of interest to consider whether Hopf bifurcation could occur for model \eqref{dclass} when $d$ is small. In this paper, we
give an answer to this question in a discrete spatial setting, and show that when the dispersal rate $d$ is small, large delay could induce Hopf bifurcation, and periodic solutions could occur for model \eqref{patc}.

Now we give two examples to illustrate the theoretical results obtained in Sections 2 and 3.
Firstly, suppose that there are $n = 9$ patches arranged and connected as in
Figure \ref{fig1}.
\begin{figure}[htbp]
\centering\includegraphics[width=0.6\textwidth]{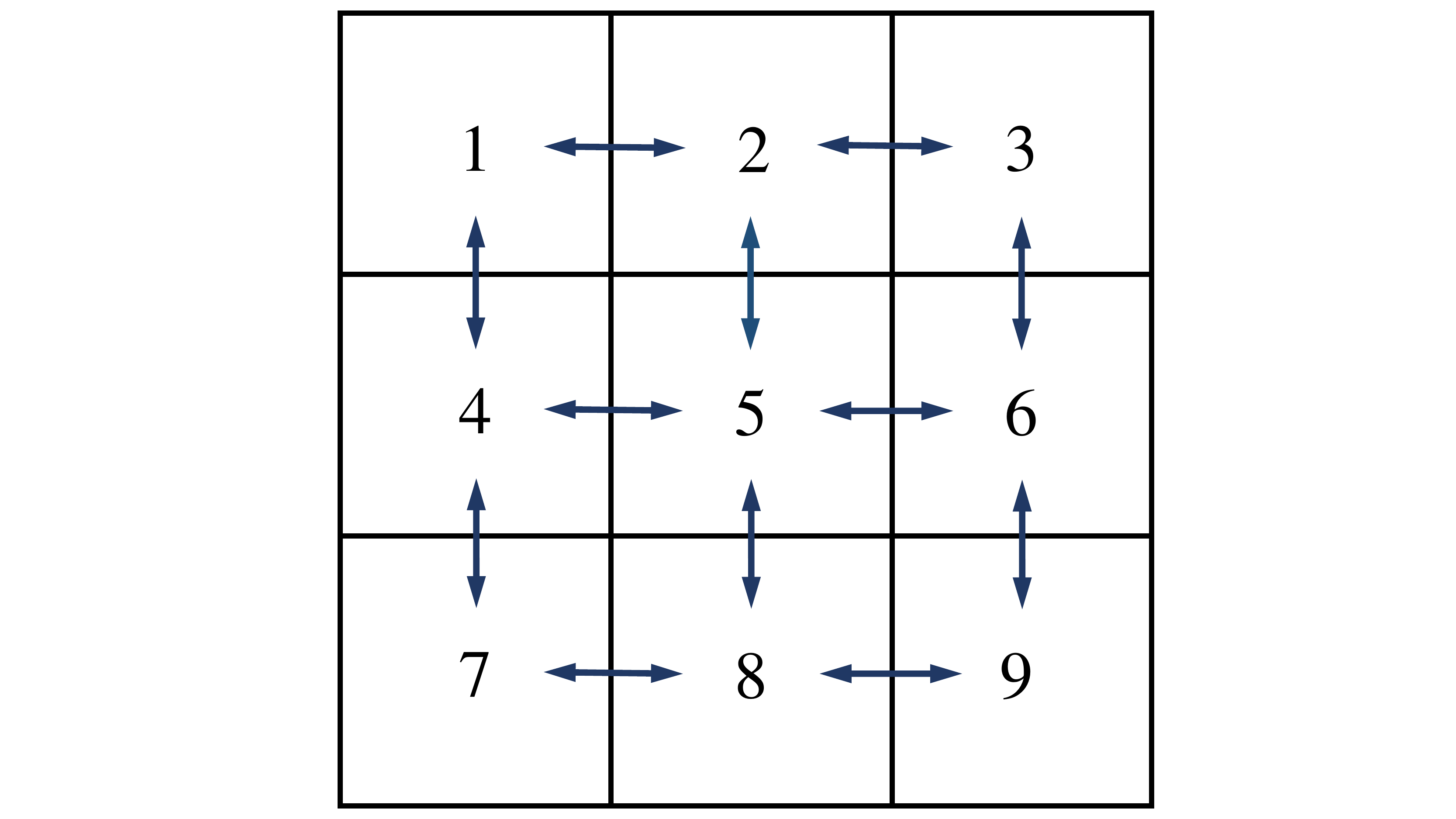}
\caption{The connectivity between nine patches.
  \label{fig1}}
\end{figure}
We choose the symmetric connectivity matrix $D=(d_{ij})$ and $(m_j)$  as follows:
\begin{equation}
\begin{split}
&d_{12}=d_{36}=d_{78}=3,\;\;d_{14}=d_{45}=d_{56}=d_{69}=2,\\
&d_{23}=d_{47}=d_{89}=1,\;\;d_{25}=4,\;\;d_{58}=5,\;\;\text{other}\;\; d_{ij} (i<j)=0,\\
&(m_1,\dots,m_9)=(10,8,16,20,24,12,18,6,14).
\end{split}
\end{equation}
Then large delay could induce Hopf bifurcation, and periodic solutions could occur when the dispersal rate $d$ is large or small, see Figure \ref{fig2}.
As is pointed out in Lemma \ref{l4.4}, there exist multiple Hopf bifurcation curves $\{r_{q0}\}_{q=1}^p=\left\{\frac{\theta^d_q}{\nu^d_q}\right\}_{q=1}^p$ when the dispersal rate $d$ is small. We conjecture that these curves could intersect when $d$ increases from zero, and consequently double Hopf bifurcation could occur and quasi-periodic solutions may arise. In fact, we numerically show the existence of such quasi-periodic solutions, see Figure \ref{fig3}.
\begin{figure}[htbp]
\centering\includegraphics[width=0.5\textwidth]{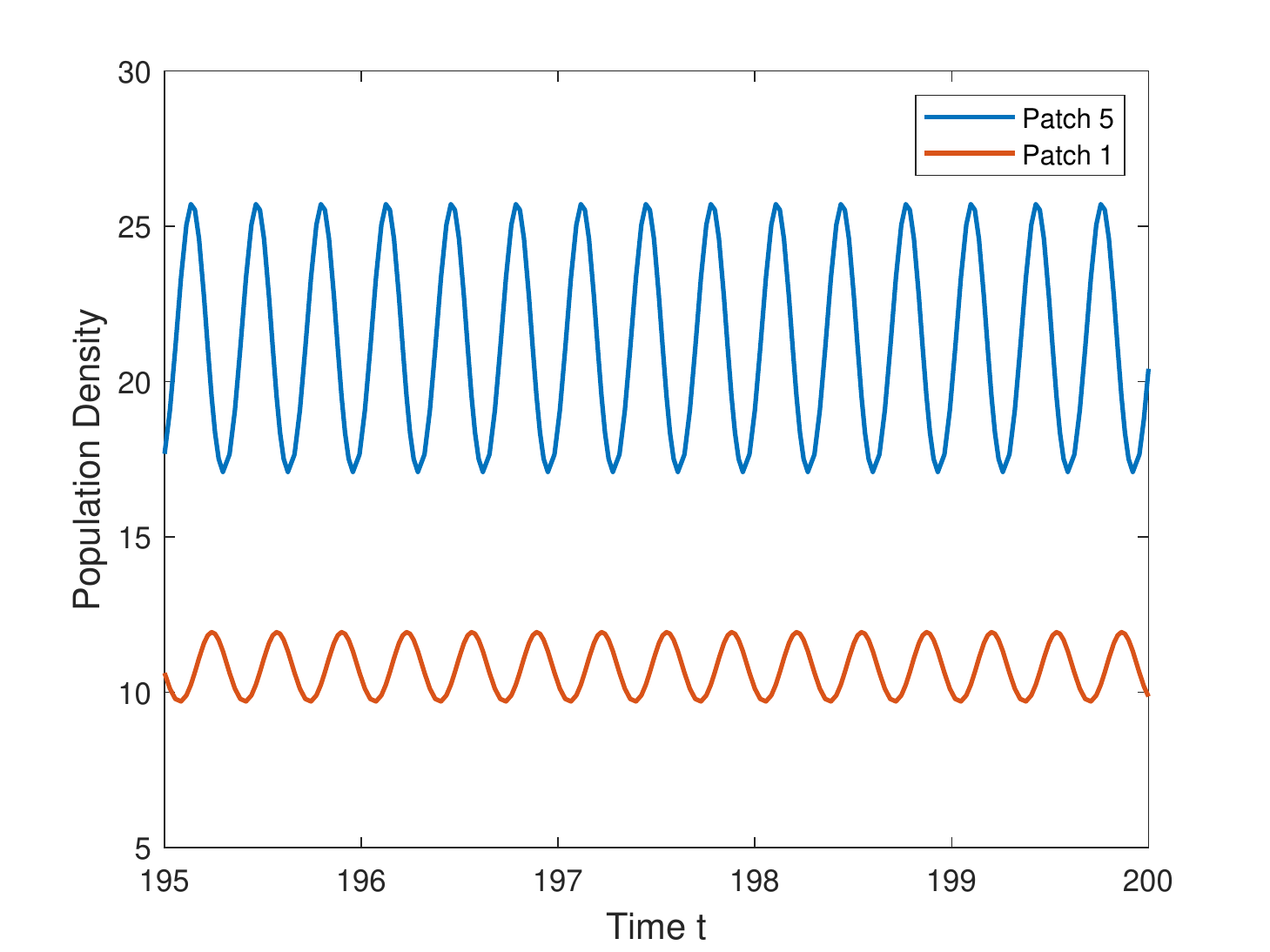}\includegraphics[width=0.5\textwidth]{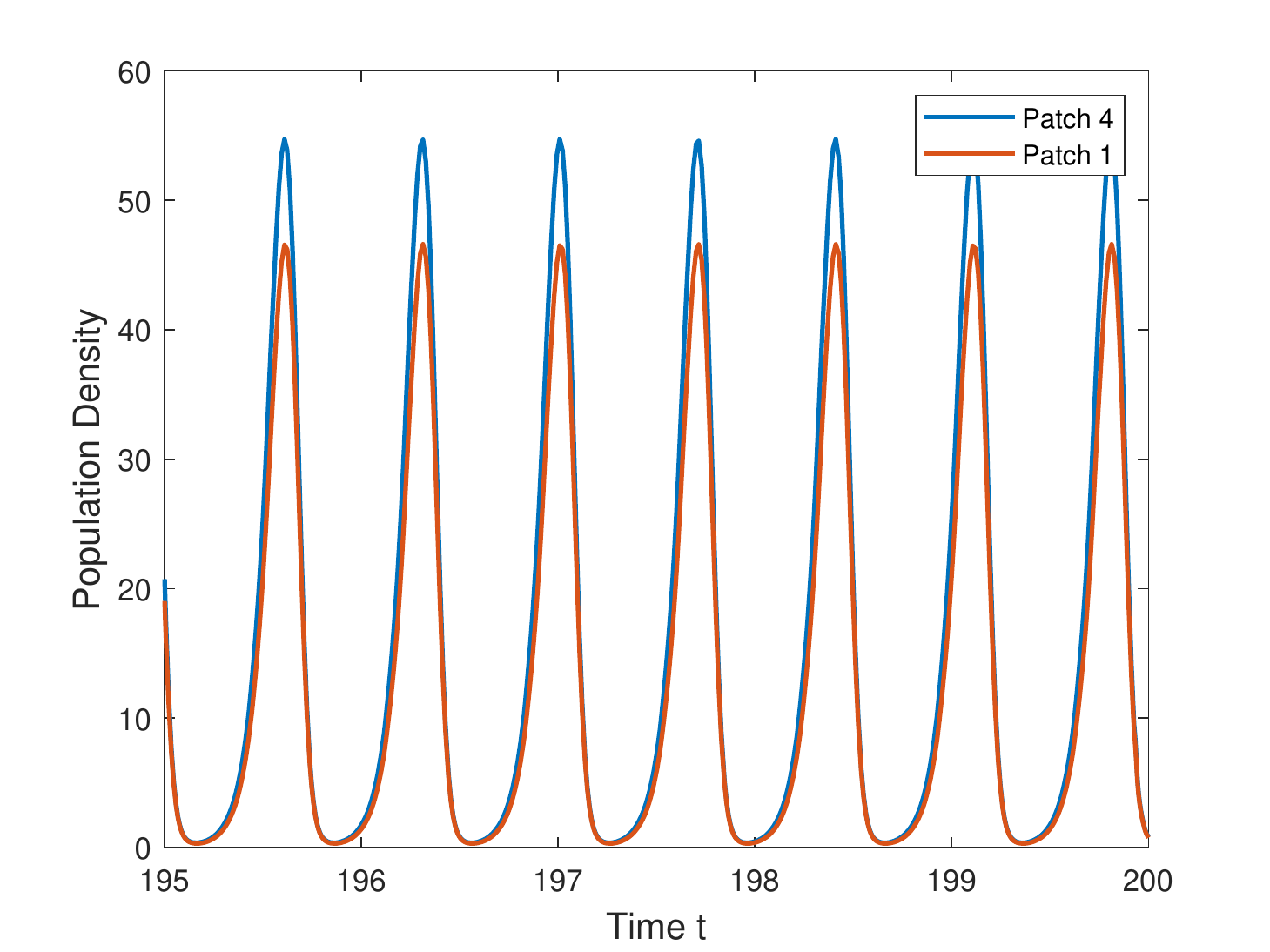}
\caption{The periodic solutions, and we only choose two patches for simplicity. (Left) the small dispersal case: $d=0.5$ and $r=0.087$. (Right) the large dispersal case: $d=10$ and $r=0.15$.
  \label{fig2}}
\end{figure}

\begin{figure}[htbp]
\centering\includegraphics[width=0.5\textwidth]{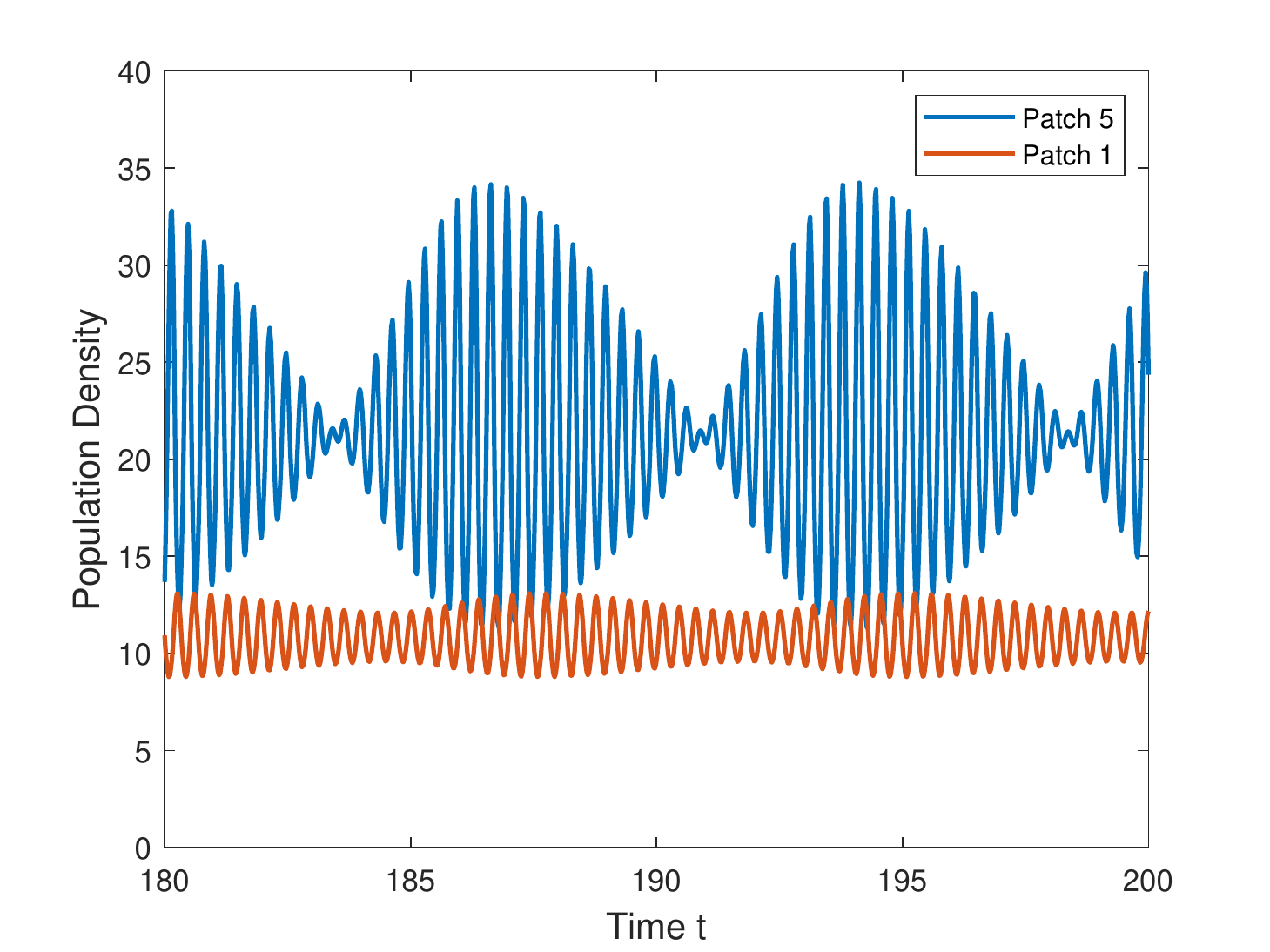}\includegraphics[width=0.5\textwidth]{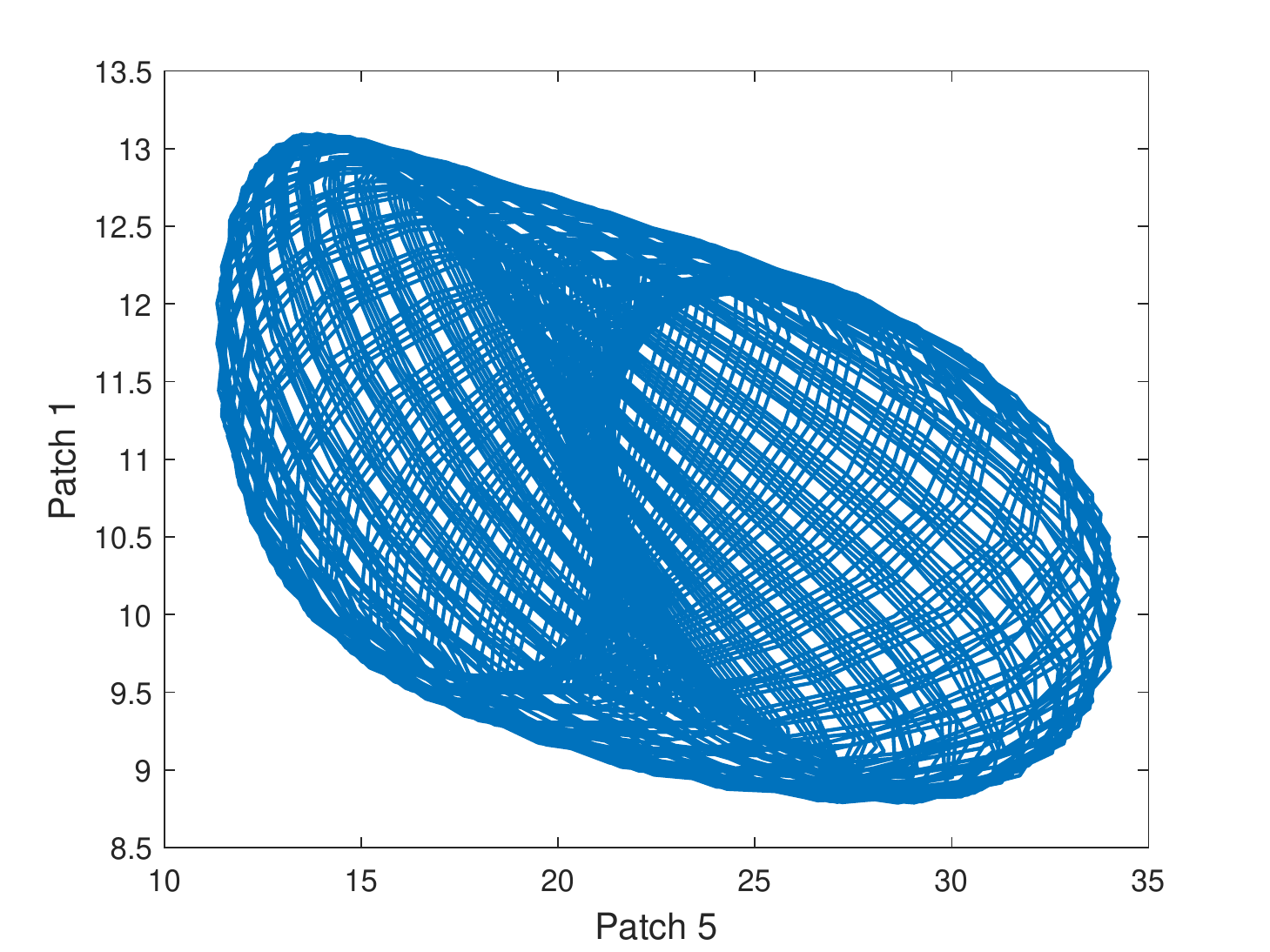}
\caption{The quasi-periodic solution. Here $d=0.5$ and $r=0.09$, and we only choose two patches for simplicity. The right one is the phase portrait for patches $1$ and $5$.
  \label{fig3}}
\end{figure}

Now we increase the patch numbers, and assume that there are $n=100$ patches arranged and connected as in
Figure \ref{fig33}.
Following \cite{GuoY2015,TianC2019}, we plot the solution of each path in one figure to show the spatiotemporal patterns. There exist
strip-typed inhomogeneous periodic solution and check-typed inhomogeneous periodic solution, see Figures \ref{fig4}-\ref{fig5}.
\begin{figure}[htbp]
\centering\includegraphics[width=0.8\textwidth]{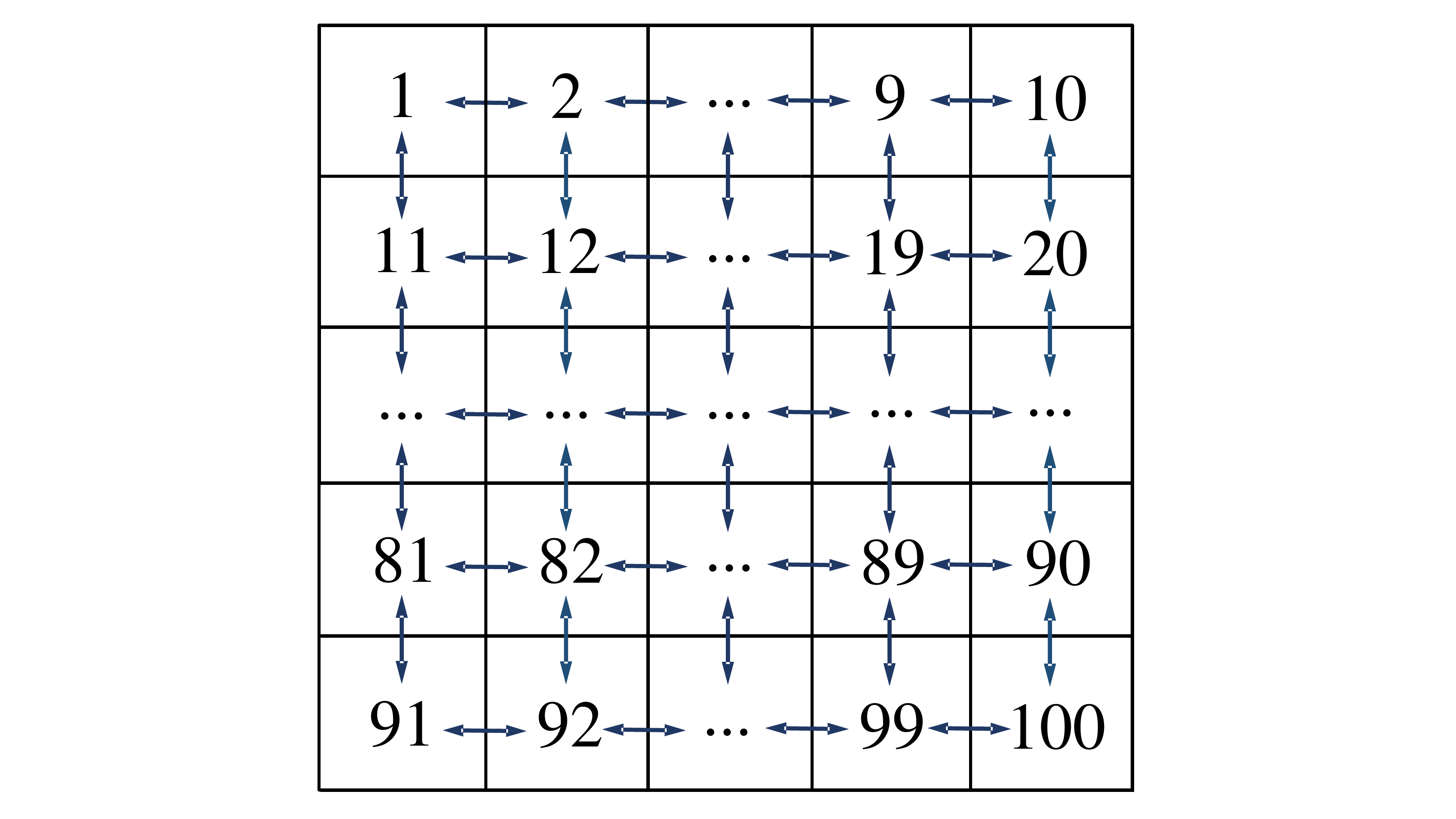}
\caption{The connectivity between one hundred patches.
  \label{fig33}}
\end{figure}

\begin{figure}[htbp]
\centering\includegraphics[width=0.5\textwidth]{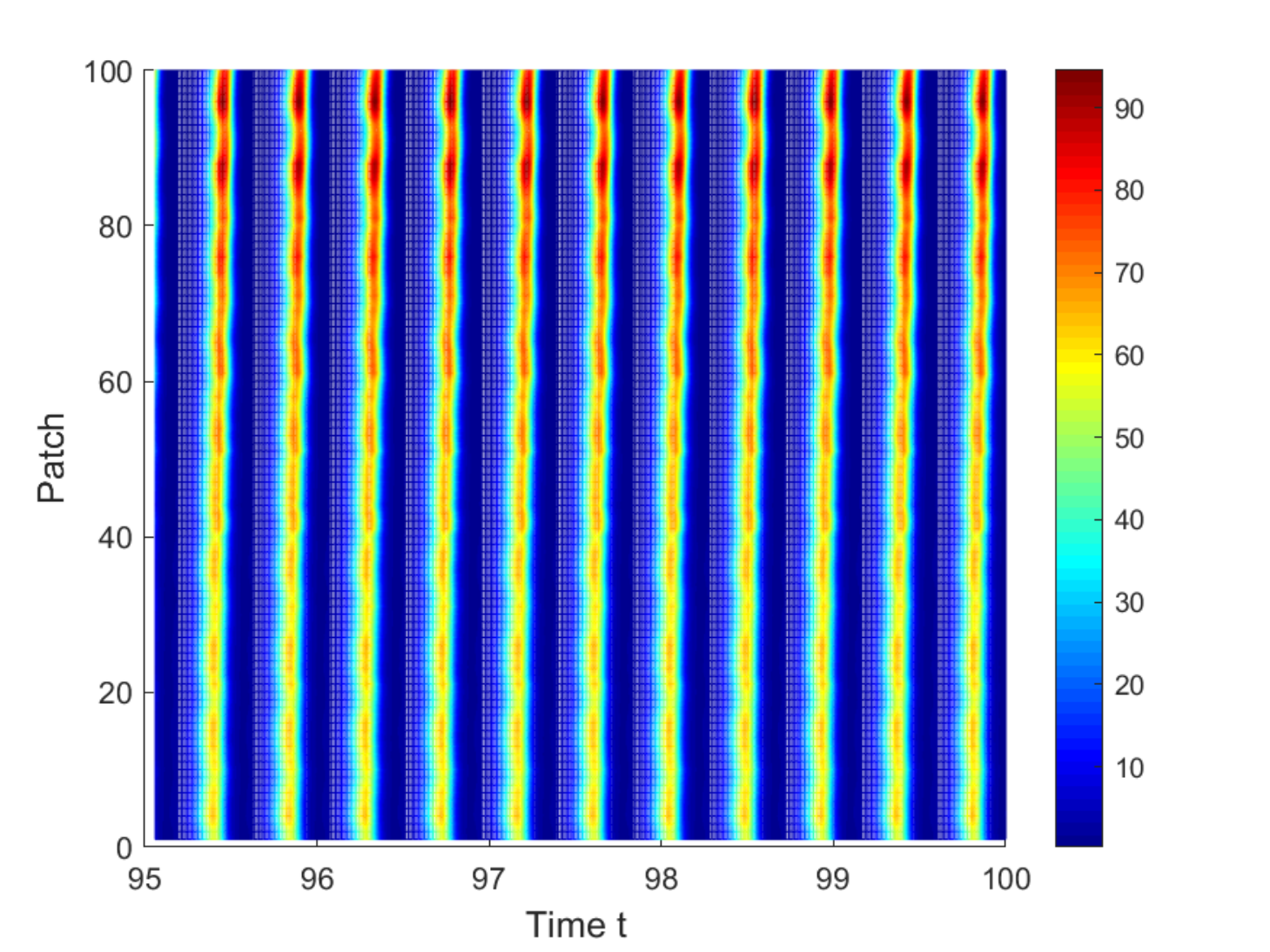}\includegraphics[width=0.5\textwidth]{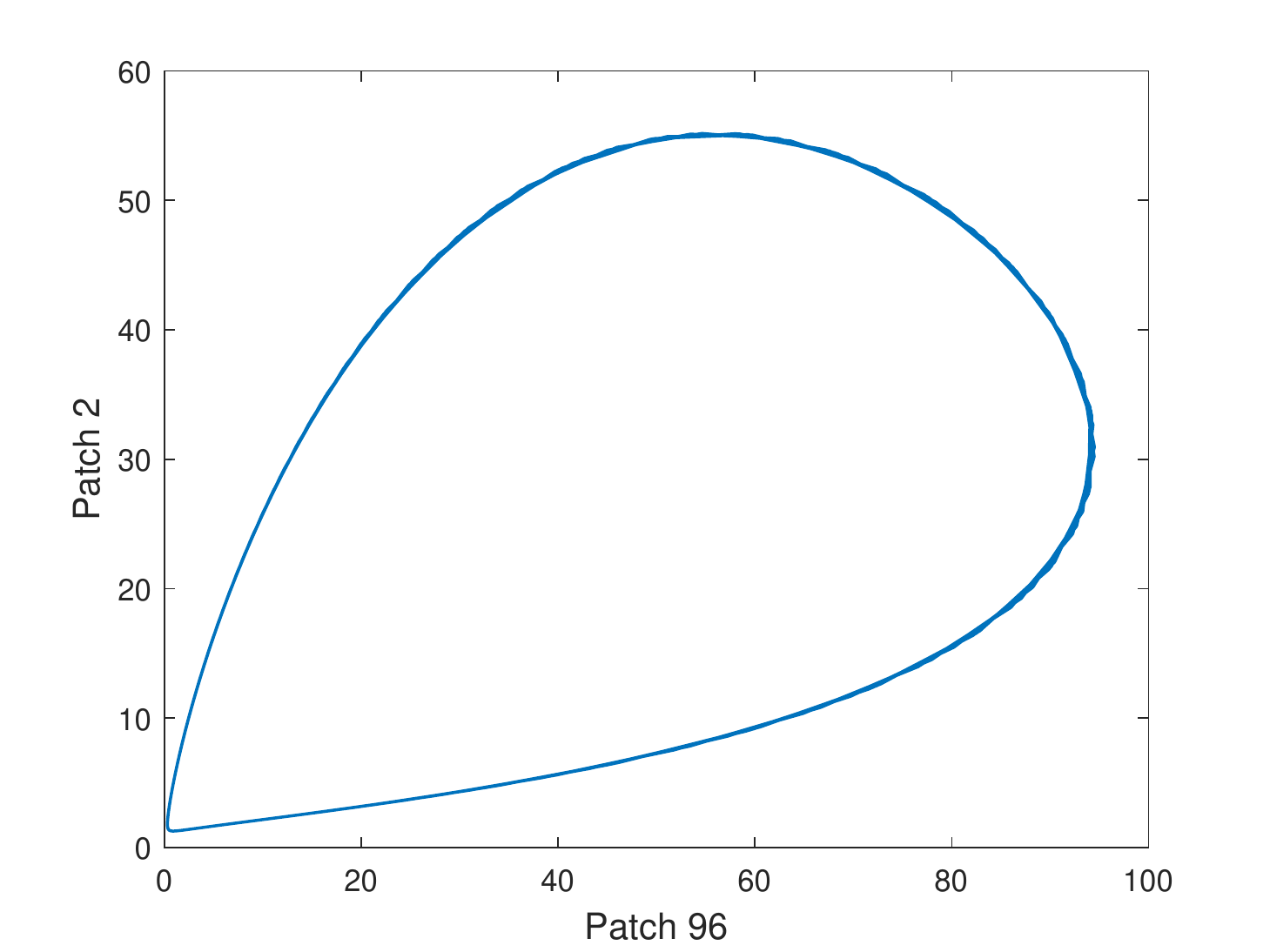}
\caption{Strip-typed inhomogeneous periodic solution. Here $d=10$ and the $y$-axis represents the patch number from $1$ to $100$. The right one is the phase portrait for patches $2$ and $96$.
  \label{fig4}}
\end{figure}

\begin{figure}[htbp]
\centering\includegraphics[width=0.5\textwidth]{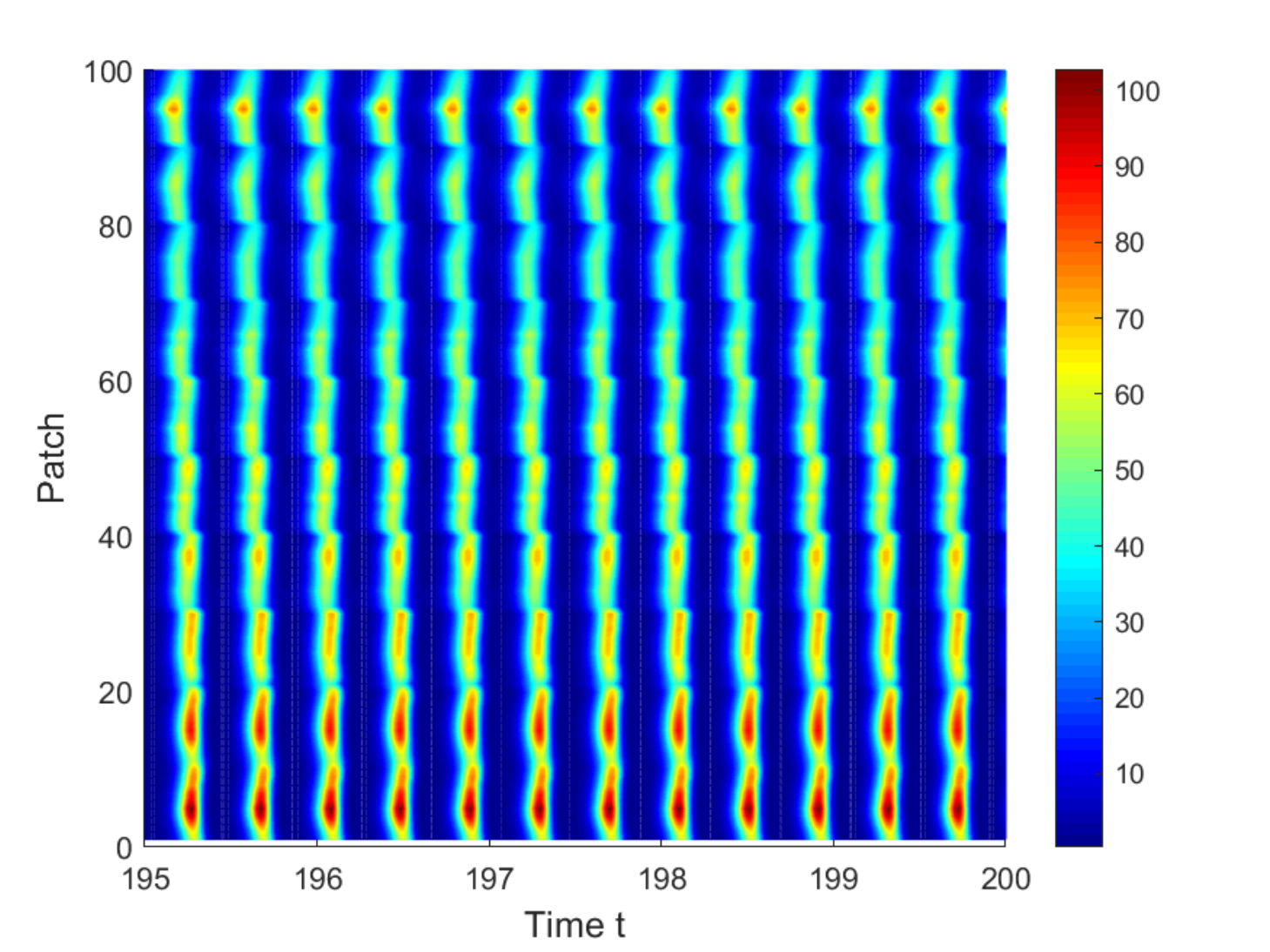}\includegraphics[width=0.5\textwidth]{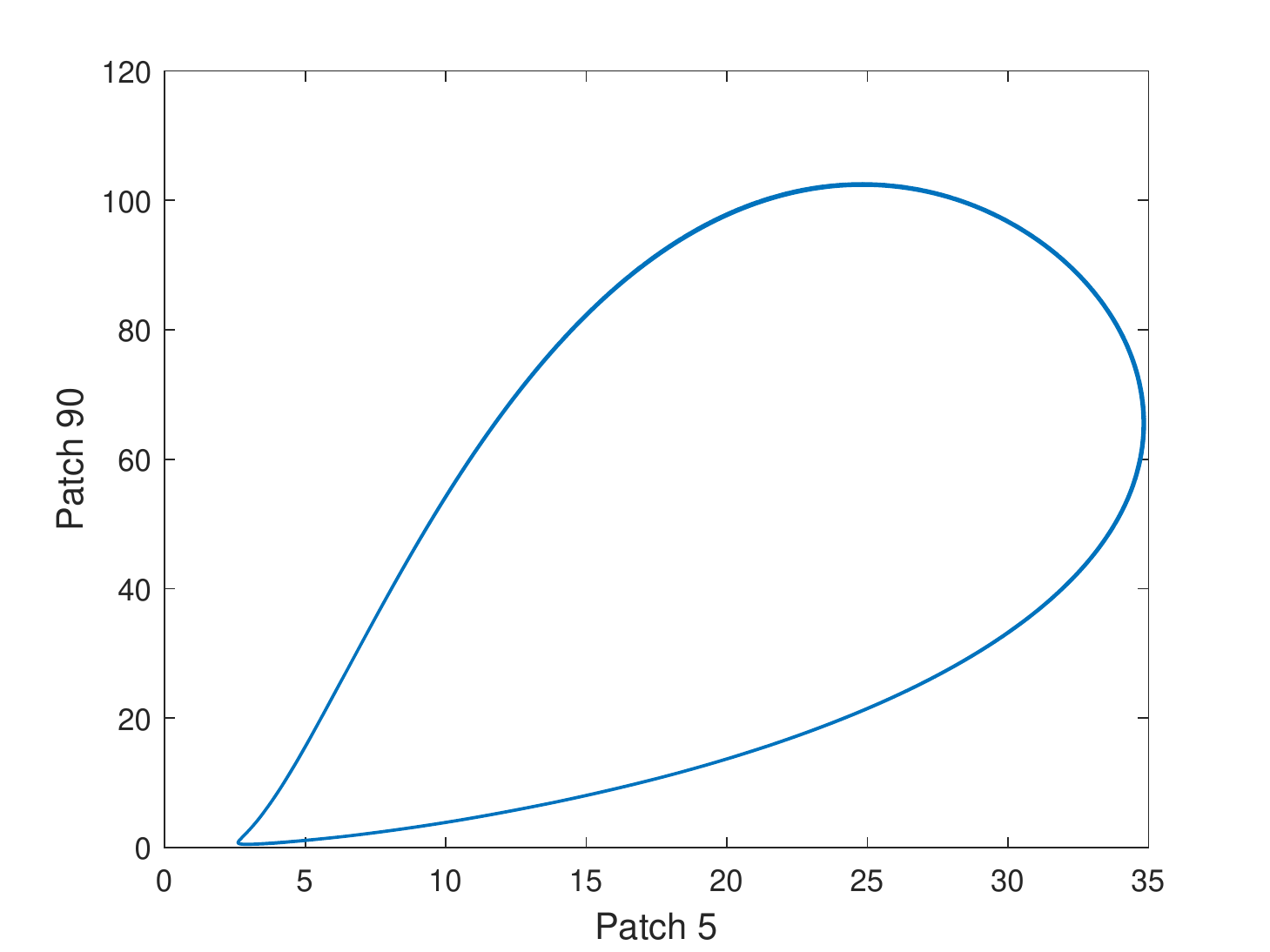}
\caption{Check-typed inhomogeneous periodic solution. Here $d=15$ and the $y$-axis represents the patch number from $1$ to $100$. The right one is the phase portrait for patches $2$ and $96$.
  \label{fig5}}
\end{figure}


\begin{thebibliography}{10}

\bibitem{Allenpatch}
L.~J.~S. Allen, B.~M. Bolker, Y.~Lou, and A.~L. Nevai.
\newblock Asymptotic profiles of the steady states for an {SIS} epidemic patch
  model.
\newblock {\em SIAM J. Appl. Math.}, 67(5):1283--1309, 2007.

\bibitem{Arditi2015}
R.~Arditi, C.~Lobry, and T.~Sari.
\newblock Is dispersal always beneficial to carrying capacity? new insights
  from the multi-patch logistic equation.
\newblock {\em Theor. Popul. Biol.}, 106:45--59, 2015.

\bibitem{Arditi2018}
R.~Arditi, C.~Lobry, and T.~Sari.
\newblock Asymmetric dispersal in the multi-patch logistic equation.
\newblock {\em Theor. Popul. Biol.}, 120:11--15, 2018.

\bibitem{Belgacem1995}
F.~Belgacem and C.~Cosner.
\newblock The effects of dispersal along environmental gradients on the
  dynamics of populations in heterogeneous environments.
\newblock {\em Canad. Appl. Math. Quart.}, 3(4):379--397, 1995.

\bibitem{Bichara2018}
D.~Bichara and A.~Iggidr.
\newblock Multi-patch and multi-group epidemic models: a new framework.
\newblock {\em J. Math. Biol.}, 77(1):107--134, 2018.

\bibitem{Busenberg}
S.~Busenberg and W.~Huang.
\newblock Stability and {H}opf bifurcation for a population delay model with
  diffusion effects.
\newblock {\em J. Differential Equations}, 124(1):80--107, 1996.

\bibitem{CantrellCosner2003}
R.~S. Cantrell and C.~Cosner.
\newblock {\em Spatial ecology via reaction-diffusion equations}.
\newblock Wiley Series in Mathematical and Computational Biology. John Wiley \&
  Sons, Ltd., Chichester, 2003.

\bibitem{ChenLouWei}
S.~Chen, Y.~Lou, and J.~Wei.
\newblock Hopf bifurcation in a delayed reaction-diffusion-advection population
  model.
\newblock {\em J. Differential Equations}, 264(8):5333--5359, 2018.

\bibitem{ChenShi2012}
S.~Chen and J.~Shi.
\newblock Stability and {H}opf bifurcation in a diffusive logistic population
  model with nonlocal delay effect.
\newblock {\em J. Differential Equations}, 253(12):3440--3470, 2012.

\bibitem{LiShuai2010}
S.~Chen, J.~Shi, Z.~Shuai, and Y.~Wu.
\newblock Spectral monotonicity of perturbed quasi-positive matrices with
  applications in population dynamics.
\newblock {\em Submitted}.

\bibitem{ChenWeiZhang}
S.~Chen, J.~Wei, and X.~Zhang.
\newblock Hopf bifurcation for a delayed diffusive logistic population model in
  the advective heterogeneous environment.
\newblock {\em to appear in J. Dynam. Differential Equations}, 2019.

\bibitem{ChenYu2016}
S.~Chen and J.~Yu.
\newblock Stability and bifurcations in a nonlocal delayed reaction-diffusion
  population model.
\newblock {\em J. Differential Equations}, 260(1):218--240, 2016.

\bibitem{CosnerLou2003}
C.~Cosner and Y.~Lou.
\newblock Does movement toward better environments always benefit a population?
\newblock {\em J. Math. Anal. Appl.}, 277(2):489--503, 2003.

\bibitem{Crandall}
M.~G. Crandall and P.~H. Rabinowitz.
\newblock Bifurcation from simple eigenvalues.
\newblock {\em J. Functional Analysis}, 8:321--340, 1971.

\bibitem{ZhangB2016}
D.~L. DeAngelis, W.-M. Ni, and B.~Zhang.
\newblock Effects of diffusion on total biomass in heterogeneous continuous and
  discrete-patch systems.
\newblock {\em Theor. Ecol.}, 9(4):443--453, 2016.

\bibitem{Gao2019}
D.~Gao.
\newblock Travel frequency and infectious diseases.
\newblock {\em SIAM J. Appl. Math.}, 79(4):1581--1606, 2019.

\bibitem{Guo2015}
S.~Guo.
\newblock Stability and bifurcation in a reaction-diffusion model with nonlocal
  delay effect.
\newblock {\em J. Differential Equations}, 259(4):1409--1448, 2015.

\bibitem{Guo2017}
S.~Guo.
\newblock Spatio-temporal patterns in a diffusive model with non-local delay
  effect.
\newblock {\em IMA J. Appl. Math.}, 82(4):864--908, 2017.

\bibitem{GuoYan2016}
S.~Guo and S.~Yan.
\newblock Hopf bifurcation in a diffusive {L}otka-{V}olterra type system with
  nonlocal delay effect.
\newblock {\em J. Differential Equations}, 260(1):781--817, 2016.

\bibitem{GuoY2015}
Y.~Guo and B.~Niu.
\newblock Amplitude death and spatiotemporal bifurcations in nonlocally
  delay-coupled oscillators.
\newblock {\em Nonlinearity}, 28(6):1841--1858, 2015.

\bibitem{Hale1971}
J.~Hale.
\newblock {\em Theory of functional differential equations}.
\newblock Springer-Verlag, New York-Heidelberg, second edition, 1977.
\newblock Applied Mathematical Sciences, Vol. 3.

\bibitem{HuYuan2011}
R.~Hu and Y.~Yuan.
\newblock Spatially nonhomogeneous equilibrium in a reaction-diffusion system
  with distributed delay.
\newblock {\em J. Differential Equations}, 250(6):2779--2806, 2011.

\bibitem{KangY2017}
Y.~Kang, S.~K. Sasmal, and K.~Messan.
\newblock A two-patch prey-predator model with predator dispersal driven by the
  predation strength.
\newblock {\em Math. Biosci. Eng.}, 14(4):843--880, 2017.

\bibitem{KuangY}
Y.~Kuang.
\newblock {\em Delay differential equations with applications in population
  dynamics}, volume 191 of {\em Mathematics in Science and Engineering}.
\newblock Academic Press, Inc., Boston, MA, 1993.

\bibitem{Lam2011}
K.-Y. Lam.
\newblock Concentration phenomena of a semilinear elliptic equation with large
  advection in an ecological model.
\newblock {\em J. Differential Equations}, 250(1):161--181, 2011.

\bibitem{LiHC2019}
H.~Li and R.~Peng.
\newblock Dynamics and asymptotic profiles of endemic equilibrium for {SIS}
  epidemic patch models.
\newblock {\em J. Math. Biol.}, 79(4):1279--1317, 2019.

\bibitem{LiaoLou2014}
K.-L. Liao and Y.~Lou.
\newblock The effect of time delay in a two-patch model with random dispersal.
\newblock {\em Bull. Math. Biol.}, 76(2):335--376, 2014.

\bibitem{LouY2006}
Y.~Lou.
\newblock On the effects of migration and spatial heterogeneity on single and
  multiple species.
\newblock {\em J. Differential Equations}, 223(2):400--426, 2006.

\bibitem{LouLutscher2014}
Y.~Lou and F.~Lutscher.
\newblock Evolution of dispersal in open advective environments.
\newblock {\em J. Math. Biol.}, 69(6-7):1319--1342, 2014.

\bibitem{LouZhou2015}
Y.~Lou and P.~Zhou.
\newblock Evolution of dispersal in advective homogeneous environment: the
  effect of boundary conditions.
\newblock {\em J. Differential Equations}, 259(1):141--171, 2015.

\bibitem{Memory}
M.~C. Memory.
\newblock Bifurcation and asymptotic behavior of solutions of a
  delay-differential equation with diffusion.
\newblock {\em SIAM J. Math. Anal.}, 20(3):533--546, 1989.

\bibitem{ShiShi2019}
Q.~Shi, J.~Shi, and Y.~Song.
\newblock Hopf bifurcation and pattern formation in a delayed diffusive
  logistic model with spatial heterogeneity.
\newblock {\em Discrete Contin. Dyn. Syst. Ser. B}, 24(2):467--486, 2019.

\bibitem{SmithW}
H.~L. Smith and P.~Waltman.
\newblock {\em The theory of the chemostat}, volume~13 of {\em Cambridge
  Studies in Mathematical Biology}.
\newblock Cambridge University Press, Cambridge, 1995.
\newblock Dynamics of microbial competition.

\bibitem{Speirs2001}
D.~C. Speirs and W.~S.~C. Gurney.
\newblock Population persistence in rivers and estuaries.
\newblock {\em Ecology}, 82(5):1219--1237, 2001.

\bibitem{SuWeiShi2009}
Y.~Su, J.~Wei, and J.~Shi.
\newblock Hopf bifurcations in a reaction-diffusion population model with delay
  effect.
\newblock {\em J. Differential Equations}, 247(4):1156--1184, 2009.

\bibitem{SuWeiShi2012}
Y.~Su, J.~Wei, and J.~Shi.
\newblock Hopf bifurcation in a diffusive logistic equation with mixed delayed
  and instantaneous density dependence.
\newblock {\em J. Dynam. Differential Equations}, 24(4):897--925, 2012.

\bibitem{TianC2019}
C.~Tian and S.~Ruan.
\newblock Pattern formation and synchronism in an allelopathic plankton model
  with delay in a network.
\newblock {\em SIAM J. Appl. Dyn. Syst.}, 18(1):531--557, 2019.

\bibitem{Vasilyeva2019}
O.~Vasilyeva.
\newblock Population dynamics in river networks: analysis of steady states.
\newblock {\em J. Math. Biol.}, 79(1):63--100, 2019.

\bibitem{Vasilyeva2010}
O.~Vasilyeva and F.~Lutscher.
\newblock Population dynamics in rivers: analysis of steady states.
\newblock {\em Can. Appl. Math. Q.}, 18(4):439--469, 2010.

\bibitem{WangXY2016}
X.~Wang and X.~Zou.
\newblock On a two-patch predator-prey model with adaptive habitancy of
  predators.
\newblock {\em Discrete Contin. Dyn. Syst. Ser. B}, 21(2):677--697, 2016.

\bibitem{YanLi2010}
X.-P. Yan and W.-T. Li.
\newblock Stability of bifurcating periodic solutions in a delayed
  reaction-diffusion population model.
\newblock {\em Nonlinearity}, 23(6):1413--1431, 2010.

\bibitem{YanLi2012}
X.-P. Yan and W.-T. Li.
\newblock Stability and {H}opf bifurcations for a delayed diffusion system in
  population dynamics.
\newblock {\em Discrete Contin. Dyn. Syst. Ser. B}, 17(1):367--399, 2012.

\bibitem{Yoshida}
K.~Yoshida.
\newblock The {H}opf bifurcation and its stability for semilinear diffusion
  equations with time delay arising in ecology.
\newblock {\em Hiroshima Math. J.}, 12(2):321--348, 1982.

\bibitem{ZhangB2015}
B.~Zhang, X.~Liu, D.~L. DeAngelis, W.-M. Ni, and G.~Geoff Wang.
\newblock Effects of dispersal on total biomass in a patchy, heterogeneous
  system: analysis and experiment.
\newblock {\em Math. Biosci.}, 264:54--62, 2015.

\bibitem{ZhuHP2018}
J.~Zhang, C.~Cosner, and H.~Zhu.
\newblock Two-patch model for the spread of {W}est {N}ile virus.
\newblock {\em Bull. Math. Biol.}, 80(4):840--863, 2018.

\bibitem{ZhaoJing}
X.-Q. Zhao and Z.-J. Jing.
\newblock Global asymptotic behavior in some cooperative systems of
  functional-differential equations.
\newblock {\em Canad. Appl. Math. Quart.}, 4(4):421--444, 1996.

\bibitem{ZhouZhao2018}
P.~Zhou and X.-Q. Zhao.
\newblock Global dynamics of a two species competition model in open stream
  environments.
\newblock {\em J. Dynam. Differential Equations}, 30(2):613--636, 2018.

\end{thebibliography}
\end{document}